\documentclass[12pt]{article}
\usepackage{amscd,amsthm,amsfonts,amssymb,amsmath,euscript}
\usepackage{graphicx}
\input epsf
\def\hybrid{\topmargin 0pt      \oddsidemargin 0pt
        \headheight 0pt \headsep 0pt
        \voffset=-0.5cm
        \textwidth 6.25in       
        \textheight 9.5in       
        \marginparwidth 0.0in
        \parskip 5pt plus 1pt   \jot = 1.5ex}
\catcode`\@=11
\def\marginnote#1{}

\newcount\hour
\newcount\minute
\newtoks\amorpm
\hour=\time\divide\hour by60
\minute=\time{\multiply\hour by60 \global\advance\minute by-\hour}
\edef\standardtime{{\ifnum\hour<12 \global\amorpm={am}%
        \else\global\amorpm={pm}\advance\hour by-12 \fi
        \ifnum\hour=0 \hour=12 \fi
        \number\hour:\ifnum\minute<10 0\fi\number\minute\the\amorpm}}
\edef\militarytime{\number\hour:\ifnum\minute<10 0\fi\number\minute}

\def\draftlabel#1{{\@bsphack\if@filesw {\let\thepage\relax
   \xdef\@gtempa{\write\@auxout{\string
      \newlabel{#1}{{\@currentlabel}{\thepage}}}}}\@gtempa
   \if@nobreak \ifvmode\nobreak\fi\fi\fi\@esphack}
        \gdef\@eqnlabel{#1}}
\def\@eqnlabel{}
\def\@vacuum{}
\def\draftmarginnote#1{\marginpar{\raggedright\scriptsize\tt#1}}
\def\draftlabel#1{{\@bsphack\if@filesw {\let\thepage\relax
   \xdef\@gtempa{\write\@auxout{\string
      \newlabel{#1}{{\@currentlabel}{\thepage}}}}}\@gtempa
   \if@nobreak \ifvmode\nobreak\fi\fi\fi\@esphack}
        \gdef\@eqnlabel{#1}}
\def\@eqnlabel{}
\def\@vacuum{}
\def\draftmarginnote#1{\marginpar{\raggedright\scriptsize\tt#1}}

\def\draft{\oddsidemargin -.5truein
        \def\@oddfoot{\sl preliminary draft \hfil
        \rm\thepage\hfil\sl\today\quad\militarytime}
        \let\@evenfoot\@oddfoot \overfullrule 3pt
        \let\label=\draftlabel
        \let\marginnote=\draftmarginnote
   \def\@eqnnum{(\theequation)\rlap{\kern\marginparsep\tt\@eqnlabel}%
\global\let\@eqnlabel\@vacuum}  }

\def\numberbysection{\@addtoreset{equation}{section}
        \def\theequation{\thesection.\arabic{equation}}}

\def\underline#1{\relax\ifmmode\@@underline#1\else
        $\@@underline{\hbox{#1}}$\relax\fi}

\def\titlepage{\@restonecolfalse\if@twocolumn\@restonecoltrue\onecolumn
     \else \newpage \fi \thispagestyle{empty}\c@page\z@
        \def\thefootnote{\fnsymbol{footnote}} }

\def\endtitlepage{\if@restonecol\twocolumn \else  \fi
        \def\thefootnote{\arabic{footnote}}
        \setcounter{footnote}{0}}  
\relax

\numberbysection
\hybrid

\newcommand{\DDD}{\raise-1pt\hbox{$\mbox{\Bbbb D}$}}

\newcommand{\UUU}{\raise-1pt\hbox{$\mbox{\Bbbb U}$}}

\newcommand{\z}{\raise-1pt\hbox{$\mbox{\Bbbb Z}$}}

\def\beq{\begin{equation}}
\def\eeq{\end{equation}}
\def\p{\partial}

\newtheorem{theorem}{Theorem}[section]
\newtheorem{lemma}{Lemma}[section]
\newtheorem{lemma-definition}{Lemma-Definition}[section]
\newtheorem{corollary}{Corollary}[section]

\newtheorem{proposition}{Proposition}[section]

\newcommand{\Cov}{\mathop{\rm Cov}\nolimits}
\newcommand{\Aut}{\mathop{\rm Aut}\nolimits}

\begin{document}
\begin{titlepage}

\title{Symmetric solutions to dispersionless 2D Toda hierarchy, Hurwitz numbers and conformal dynamics}

\author{S.M.~Natanzon\thanks{National Research University Higher School of Economics, 20 Myasnitskaya Ulitsa, Moscow 101000, Russia and A.N.Belozersky Institute, Moscow State University, Moscow, Russia and ITEP, Moscow, Russia,
e-mail: natanzons@mail.ru}
\and A.V.~Zabrodin
\thanks{Institute of Biochemical Physics,
4 Kosygina st., Moscow 119334, Russia; ITEP, 25
B.Cheremushkinskaya, Moscow 117218, Russia and
National Research University Higher School of Economics,
20 Myasnitskaya Ulitsa,
Moscow 101000, Russia, e-mail: zabrodin@itep.ru}}

\date{February 2013}
\maketitle

\begin{abstract}

We explicitly construct the series expansion for a certain class of
solutions to the 2D Toda hierarchy in the zero dispersion limit,
which we call symmetric solutions. We express
the Taylor coefficients through some universal combinatorial
constants and find recurrence relations for them.
These results are used to obtain new formulas for the
genus 0 double Hurwitz numbers. They can also serve
as a starting point for a constructive approach
to the Riemann mapping problem and the inverse potential problem in 2D.

\end{abstract}

\vfill

\end{titlepage}

\tableofcontents

\section{Introduction}

\vspace{1ex}

The dispersionless 2D Toda lattice (2DTL) hierarchy  was introduced
by Takasaki and Takebe \cite{TakTak1,TakTak}. It can be represented in two
equivalent ways: in the Lax-Sato form or in the Hirota form.
In this paper we use the latter formulation.
The dispersionless 2DTL hierarchy
is an infinite system of differential equations
\beq\label{eq1}(z-\xi)\exp\bigl (D(z)D(\xi)F\bigr )=z \exp \bigl (
-\partial_0 D(z)F\bigr )-\xi
\exp\bigl (-\partial_0D(\xi)F\bigr ),
\eeq
\beq\label{eq2}(\bar z-\bar\xi)\exp\bigl ( \bar D(\bar z)
\bar D(\bar\xi)F\bigr )=
\bar z \exp\bigl (-\partial_0\bar D(\bar z)F\bigr )-\bar\xi
\exp\bigl (-\partial_0\bar D(\bar\xi)F\bigr ),
\eeq
\beq\label{eq3}
1- \exp\bigl (-D(z)\bar D(\bar\xi)F\bigr )=\frac{1}{z\bar\xi}
\exp\bigl (\partial_0
(\partial_0+D(z)+\bar D(\bar\xi))F\bigr )
\eeq
for the function $F=F(t_0, {\bf t} , {\bf \bar t})$, where
${\bf t}=\{t_1, t_2, \ldots \}$, ${\bf \bar t}
=\{\bar t_1, \bar t_2, \ldots \}$
are two infinite sets of time variables and
\beq\label{DD}
D(z)=\sum_{k\geq 1}\frac{z^{-k}}{k}\, \p_{k}\,,
\quad \quad
\bar D(\bar z)=\sum_{k\geq 1}\frac{\bar z^{-k}}{k}\, \bar \p_{k}\,.
\eeq
Hereafter we abbreviate $\p_k =\p /\p t_k$, $\bar \p_k =\p/\p \bar t_k$.
Differential equations of the hierarchy are obtained by expanding
equations (\ref{eq1})-(\ref{eq3}) in powers of $z, \xi$.
These equations are known to be connected with different
branches of mathematics and mathematical physics.

Solutions to equations (\ref{eq1})-(\ref{eq3}) such that
$\partial_kF|_{t_0}=\bar{\partial}_kF|_{t_0}=0$ for all $k\geq 1$
form an especially important class.
(By $g|_{t_0}(t_0)$ we denote the restriction
of a function $g(t_0,t_1, \bar t_1, t_2,\bar t_2, \dots)$ to
the line $t_1=\bar t_1 =t_2=\bar t_2 =\ldots =0$.)
We call them \textit{symmetric solutions}.
This class contains, in particular, the $c=1$ string solution
\cite{DMP,EK,HOP,Takstring}.

\vspace{1ex}

In Section 2 we describe all formal symmetric solutions of
the dispersionless 2DTL hierarchy in the form of a Taylor series.
We prove that they are fully defined by
the restriction $F|_{t_0}$ (a function of one variable).
The Taylor coefficients are expressed
through some universal combinatorial constants $N_{(\Delta|\bar{\Delta})}\left(\begin{matrix} s_1\ldots s_m \\ r_1\ldots r_m \end{matrix}\right)$ which depend on two
Young diagrams $\Delta,\bar{\Delta}$ and two sequences
of natural numbers $\{s_i\}, \{r_i\}$.
Moreover, we find recurrence formulas for $N_{(\Delta|\bar{\Delta})} \left(\begin{matrix} s_1\ldots s_m \\ r_1\ldots r_m \end{matrix}\right)$. Our method is an
extension of the method developed in \cite{N,N03} for the formal
$c=1$ string solution. Convergence of the Taylor series for this case was investigated in \cite{KKN}.
\vspace{1ex}

In Section 3 we apply these results to the double Hurwitz numbers for coverings of genus 0. This application is based on the fact that the generating function for connected double Hurwitz numbers of genus 0 is a symmetric solution of the dispersionless 2DTL hierarchy, which follows from Okounkov's result \cite{Okounkov00} and its dispersionless limit discussed by Takasaki \cite{Takasaki12}.

\vspace{1ex}

The connected  genus 0 double Hurwitz number $H_0(\Delta|\bar{\Delta})$ is the number of non-equivalent
rational functions with pre-images of 0 and $\infty$
of fixed topological types given by Young diagrams
$\Delta =[k_1,\ldots,k_\ell ]$ and $\bar{\Delta}=[\bar k_1,\ldots, \bar k_{\bar \ell}]$ (with the condition on degrees $|\Delta |=|\bar \Delta |$)
and some fixed simple critical values (finite and non-zero). For rational functions of degree $d$ with simple finite critical values this number was found by Hurwitz \cite{Hur}:
$$
H_0([2,\underbrace{1,\ldots , 1}_{d-2}]|[k_1,\dots,k_n])=
H_0([\underbrace{1, \ldots ,1}_{d}]|[k_1,\dots,k_n])=
\frac{(d+n-2)!} {\sigma([k_1,\dots,k_n ])}
\prod\limits_{i=1}^n\frac{k_i^{k_i}}{k_i!}d^{n-3}
$$
where $\sum_{i}k_i =d$ and $\sigma(\Delta)$
is order of the automorphism group of rows for the Young diagram $\Delta$.

Further results are based on the moduli space methods \cite{Lando}.
In \cite{Lando-Zvonkine [43 Lando]} it was found, in particular, that
$$H_0([n]|\Delta )=
\frac{(\ell -1)!}{\sigma(\Delta )}\, n^{\ell -2}.
$$
Some formulas for $H_0([n_1,n_2]|\Delta )$ and
$H_0([n_1,n_2,n_3]| \Delta )$ were found in \cite{GJV,SSV}.
It was also proved that the
Hurwitz numbers $H_0([[k_1,\dots,k_n]]|[\bar{k}_1,\dots, \bar{k}_{\bar{n}}])$ depend piecewise polynomially on $k_1,\dots,k_n,\bar{k}_1,\dots, \bar{k}_{\bar{n}}$ \cite{GJV}. The domains of polynomiality (chambers)
were characterized and relations between
polynomials in neighboring chambers were found \cite{SSV}.

\vspace{1ex}

In Section 3, by comparing the generating function
for the genus 0 double Hurwitz numbers with the formulas
for symmetric solutions,
we give a representation of the Hurwitz numbers through
the combinatorial coefficients
$N_{(\Delta|\bar{\Delta})}\left(\begin{matrix} s_1\ldots s_m \\ r_1\ldots r_m \end{matrix}\right)$. These formulas appear to be new and useful.
In particular, they make
explicit some general properties of the genus 0 double
Hurwitz numbers (for example, the piecewise polynomiality).
Moreover, they are well-suited for direct calculations with the
help of the computer and allow one
to obtain closed formulas for $H_0([k_1,
\ldots ,k_n]| [\bar{k}_1,\ldots ,\bar{k}_{\bar{n}}])$
for any $n,\bar n$.
We illustrate this method by two examples: $H_0([n]|[k_1,\ldots ,k_m])$ and $H_0([k_1,k_2]| [\bar{k}_1, \bar{k}_2])$,
where the calculations
can be done by hands. Other consequences
of our general formulas will be discussed elsewhere.

\vspace{1ex}

Another application, addressed in Section 4,
is the class of problems which we call
conformal dynamics.
In papers \cite{MWWZ,WZ,KKMWZ,MWZ}
it was shown that some problems of complex
analysis in 2D, such as conformal mapping,
Dirichlet boundary value problem
and 2D inverse potential problem have a hidden integrable
structure, which, for simply-connected domains,
is the dispersionless 2DTL hierarchy.
In a more general case it is the universal Whitham hierarchy
introduced by Krichever
in \cite{KriW1,KriW}. The hierarchical times $t_k$ are suitably
defined harmonic moments of the domain and their complex
conjugates, with $t_0$ being proportional
to the area of the domain. One can construct the
dispersionless tau-function $F=F(t_0,\{t_k\},
\{\bar t_k\})$ which contains all the information
about the conformal bijection of any domain with given moments
to the unit disk. The function $F$ satisfies the dispersionless version
of the Hirota equations for the 2DTL hierarchy given by
(\ref{eq1})-(\ref{eq3}).

\vspace{1ex}

Further, in \cite{Ztmf} it was argued that any non-degenerate
solution of the hierarchy, with certain
reality conditions imposed, can be given a
similar geometric meaning. Such solutions are parameterized by
a function $\sigma (z, \bar z)$
of two variables which has the meaning of a density, or conformal metric,
in the complex plane. The moments should be now defined as integrals
of powers of $z$ with this density. The integral representation for the
dispersionless tau-function also changes accordingly but the formulas
which express the conformal map through its second order derivatives
do not depend on $\sigma$. In other words, the Toda dynamics encodes
the shape dependence of the conformal mapping, which we call the
{\it conformal dynamics}.
In the context of the
conformal dynamics, the symmetric solutions correspond to the axially
symmetric functions $\sigma$ (i.e., depending only on $|z|^2$).

\vspace{1ex}

Some important examples of conformal dynamics on symmetric background
are considered in Section 5.
Using the approach
of Section 2, one can write explicit formulas for the
corresponding symmetric solutions of the dispersionless 2DTL hierarchy.
The corresponding conformal dynamics provides an effectivization of the
Riemann mapping theorem.
Namely, for any domain characterized by its moments, our formulas
allow one to find the conformal map from this domain to the unit disk with
any given precision.

\vspace{2ex}

\section{Formal symmetric solutions of dispersionless 2D Toda hierarchy}

\vspace{1ex}

Let $\Delta=[\mu_1,\mu_2,\ldots,\mu_\ell ]$ be
the Young diagram with $\ell =\ell (\Delta )$
rows of non-zero lengths $\mu_1\geq\mu_2
\geq\ldots\geq\mu_\ell >0$,
and similarly for $\bar \Delta=[\bar \mu_1,\bar \mu_2,\ldots,\bar
\mu_{\bar \ell}]$, with $\bar \ell = \ell (\bar \Delta )$.
We identify $\Delta$ with the partition of the number
$|\Delta |:=\mu_1 +\ldots +
\mu_{\ell}$ into the $\ell$ non-zero parts $\mu_i$.
Another convenient notation is $\Delta =(1^{m_1}2^{m_2}\ldots
r^{m_r}\ldots )$, which means that exactly $m_i$ parts of the
partition $\Delta$ have length $i$: $m_i=\mbox{card}\, \{j: \mu_j=i\}$. Put $\sigma (\Delta)=m_1!\ldots m_{\ell}!$ and $\rho(\Delta)=\mu_1\ldots
\mu_{\ell (\Delta )}$.

\vspace{1ex}

Given the two sets of time variables,
${\bf t}=\{t_1, t_2, \ldots \}$, ${\bf \bar t}
=\{\bar t_1, \bar t_2, \ldots \}$ as before,
set $t_{\Delta}=t_{\mu_1}\dots t_{\mu_\ell }$, $\bar{t}_{\Delta}=
\bar{t}_{\mu_1}\dots \bar{t}_{\mu_\ell }$. For the empty diagram we put $t_{\emptyset}=\bar t_{\emptyset}=1$. We say that a Taylor series of the form
\beq\label{formal}
F(t_0,{\bf t},{\bf \bar{t}})=\sum\limits_{\Delta,\bar{\Delta}}
F (\Delta|\bar{\Delta}|t_0) \, t_{\Delta} \bar{t}_{\bar{\Delta}},
\eeq
where the sum is over all pairs of Young diagrams including empty ones, is a \textit{formal solution} to the dispersionless 2DTL hierarchy if
its substitution to equations (\ref{eq1})-(\ref{eq3}) gives the identity.
In a more explicit but less compact notation the series (\ref{formal}) reads
$$F(t_0,{\bf t},{\bf \bar{t}})=F (t_0) \ + \!\!
\sum\limits_{{\mu_1\geq\mu_2
\geq\ldots\geq\mu_\ell \atop \bar{\mu}_1\geq\bar{\mu}_2
\geq\ldots\geq\bar{\mu}_{\bar{\ell}}}}
F (\mu_1,\mu_2,\dots,\mu_\ell |\bar{\mu}_1,
\bar{\mu}_2,\dots, \bar{\mu}_{\bar{\ell}}|t_0) t_{\mu_1}t_{\mu_2}\dots t_{\mu_k}\bar{t}_{\bar{\mu}_1}\bar{t}_{\bar{\mu}_2}\dots
\bar{t}_{\mu_{\bar{\ell}}}.$$
Here $F (t_0)=F(t_0,0,0)$.

\vspace{1ex}

\noindent
{\bf Remark.} In this section the
bar does not mean the complex conjugation
and $t_k$, $\bar t_k$ are regarded
as arbitrary formal variables.

\subsection{Taylor expansion of symmetric solutions}

Let $g|_{t_0}(t_0)$ denote the restriction
of a function $g(t_0,t_1, \bar t_1, t_2,\bar t_2, \dots)$ to
the line $t_1=\bar t_1 =t_2=\bar t_2=\ldots =0$.
Formal solutions such that
$\partial_kF|_{t_0}=\bar{\partial}_kF|_{t_0}=0$ for all $k\geq 1$
will be called \textit{symmetric}.

\vspace{1ex}

Our first goal is to prove that symmetric formal solutions
are fully determined by
the function $F|_{t_0}(t_0):=F(t_0)$ which can be an arbitrary
twice differentiable function of $t_0$.
Moreover, we prove that $\Phi(\Delta|\bar{\Delta}|t_0)$ is a
differential polynomial in $f(t_0)=\exp(F|_{t_0}''(t_0))$
(i.e., a polynomial in
$f(t_0), f'(t_0), f''(t_0), \ldots $)
with universal coefficients which will be found.
For a particular function $f(t_0)$ they were found in \cite{N}.
In the general case the arguments are similar.
We split the proof into several lemmas.

\begin{lemma-definition} (\cite{N}, Lemma 3.3) \label{ld2.1}
For any formal solution $F$
the following relation holds
$$\partial_i\partial_j F= \!\!
\sum_{{m>0 \atop p_1+\dotsb+p_m=i+j}}
\frac{ij}{p_1\ldots p_m}\,
T_{ij}(p_1, \ldots , p_m) \, \partial_0
\partial_{p_1} F\ldots \partial_0\partial_{p_m} F,$$
where
$$
T_{ij}(p_1, \ldots , p_m)=\!\! \sum\limits_{{k>0, \, n_i>0\atop
n_1+\dots+ n_k = m}}
\frac{(-1)^{m+1}}{k \, n_1!\ldots n_k!} \,
P_{ij}\left (\sum\limits_{i=1}^{n_1}p_i,
\sum\limits_{i=n_1+1}^{n_1+n_2}\!\! p_i,\,\,
\ldots ,
\! \sum\limits_{i=n_{1}+\ldots +n_{k-1}+1}^{m}p_i\right )
$$
and $P_{ij}(r_1, \ldots ,r_m)$ is the number of
sequences of positive integers
$(i_1,\dots,i_m)$, $(j_1,\dots,j_m)$ such that $i_1+\ldots+i_m=i$, $j_1+\ldots+j_m=j$ and $r_k=i_k+j_k$.
\end{lemma-definition}

An analog of this lemma for
the dispersionless KP hierarchy can be found in  \cite{N01}.
Here and below all indices and variables with indices
like $i,p,s,\ell,n,r,a,b$ with or without bar are positive integers.

\vspace{1ex}

By induction one can prove

\begin{lemma-definition} (\cite{N}, Lemma 3.4.) \label{ld2.2}
For any formal solution $F$ and $k>1$
the following relation holds
$$\partial_{i_1}\partial_{i_2}\dotsb\partial_{i_k}F=
\sum\limits_{m=1}^\infty \Bigl (\sum \limits_{{s_1+ \dotsb+s_m=i_1+
\ldots +i_k \atop \ell_1+\dotsb +\ell_m=m+k-2}} \frac{i_1\ldots i_k}{s_1\ldots s_m}
T_{i_1\ldots i_k}\!
\left(\begin{matrix} s_1\ldots s_m \\ \ell_1\ldots \ell_m\end{matrix} \right)\partial^{\ell_1}_0\partial_{s_1} F\ldots \partial^{\ell_m}_0 \partial_{s_m} F\Bigr),
$$
where

$$ \quad
T_{i_1i_2}\left(\begin{matrix} s_1\ldots s_m \\ \ell_1 \ldots \ell_m\end{matrix}\right)=\left \{
\begin{matrix}T_{i_1i_2}(s_1,\ldots ,s_m),& \ \text{if}\
\ell_1=\ldots= \ell_m=1 \\ \\ 0 & \ \text{otherwise,}
\end{matrix}\right.$$

\vspace{1ex}

$$\begin{array}{lll}
T_{i_1\ldots i_k}\left(\begin{matrix} s_1\ldots s_m \\ \ell_1\ldots \ell_m\end{matrix}\right)&=&\displaystyle{
\sum\limits_{{1\leqslant i\leqslant j\leqslant m}} \frac{\ell!}{(\ell_i-1)!
\ldots (\ell_j-1)!}}
\\ &&\\
&\times&\displaystyle{T_{i_1\ldots i_{k-1}}\biggl( {\begin{matrix}
s_1\ldots s_{i-1} \\ \ell_1\ldots \ell_{i-1}\end{matrix}}\,
\begin{matrix} s \\ \ell \end{matrix} \, {\begin{matrix} s_{j+1}\ldots s_m\\ \ell_{j+1}\ldots \ell_m\end{matrix}} \biggr)
\, T_{s,i_k}(s_i, s_{i+1},\ldots ,s_j),}
\end{array}$$
with
$$\quad s=s_i+s_{i+1}+\dotsb +s_j-i_k>0, \quad \ell=(\ell_i-1)+\dotsb+(\ell_j-1)>0.
$$
\end{lemma-definition}

\vspace{1ex}

From now on we consider only symmetric formal solutions $F$.
Put \textbf{$f:=\exp(F|_{t_0}'')$}.

\begin{lemma}\label{symmetric} If $F$ is a symmetric formal solution, then
$$\partial_i\bar\partial_j F\vert_{t_0}=\biggl\{ \begin{matrix}
0 &\ \text{for}\ i\ne j\,,\\ i f^i &\ \text{for}\ i=j\,. \end{matrix}$$
\end{lemma}

\begin{proof} From
$\partial_0\partial_k F\bigl\vert_{t_0}=
\partial_0\bar{\partial}_k F\bigl\vert_{t_0}=0$ for $k>0$
it follows that
$$\exp(\partial_0(\partial_0+D(z)+\bar D(\bar\xi))F)
\Bigl\vert_{t_0}=\exp (\p_0^2 F|_{t_0})=\exp(F|_{t_0}'')=f.$$
Moreover, from
$1-e^{-D(z)\bar D(\bar\xi)F}=
z^{-1}\bar\xi^{-1}e^{\partial_0(\partial_0+D(z)+
\bar D(\bar\xi))F}$ we have
$$-D(z)\bar D(\bar\xi) F\bigl\vert_{t_0}=\log \left (1-z^{-1}\bar\xi^{-1}
f\right )=-\sum\limits_{k=1}^\infty \frac{1}{k} z^{-k}\bar{\xi}^{-k} f^k.$$
Therefore, $\partial_i\bar\partial_jF\bigl\vert_{t_0}=0$ for $i\ne j$ and $\partial_i\bar\partial_i F\bigl\vert_{t_0}=i f^i$.
\end{proof}

\begin{lemma}\label{l2.2} The following relations hold
$$\partial_{i}\bar\partial_{i_1}\ldots\bar\partial_{i_k}F\Bigl \vert_{t_0}=\bar\partial_i\partial_{i_1}\ldots \partial_{i_k}
F\Bigl\vert_{t_0}= \biggl\{ \begin{matrix} 0 &\ \text{if} \quad i_1 +\ldots+i_k\ne i \\i_1\ldots i_k\partial_0^{k-1}(f^i)
& \ \text{if}\quad i_1+\ldots +i_k =i.
\end{matrix}$$
\end{lemma}

\begin{proof} The differentials $\partial$ and
$\bar\partial$ enter the Toda equations in a symmetric way.
This gives the first equality.
Moreover, according to Lemma-Definition \ref{l2.2}, we have
$$\bar{\partial}_i\partial_{i_1}\partial_{i_2}\ldots \partial_{i_k} F= \frac{i_1\ldots i_k}{i_1 +\ldots +i_k}\, \partial_0^{k-1} \bar{\partial}_i\partial_{i_1 +\ldots +i_k} F
$$
$$+\,\, \bar{\partial}_i\sum\limits_{m=2}^\infty
\sum\limits_{{s_1+\ldots +s_m=i_1+\ldots +i_k
\atop \ell_1+\ldots \ell_m=m+k-2}}
\frac{i_1\ldots i_k}{s_1\ldots s_m} \,T_{i_1\ldots i_k}\left(\begin{matrix} s_1\ldots s_m \\ \ell_1\ldots\ell_m\end{matrix}\right)
\partial^{\ell_1}_0\partial_{s_1} F\ldots \partial^{\ell_m}_0
\partial_{s_m} F.
$$
This equality and Lemma \ref{symmetric} gives the second equality in the assertion of Lemma \ref{l2.2}.
\end{proof}

\vspace{2ex}

Given a sequence of natural numbers $\{a_1,\dots,a_v\}$, we say that its representation as a union of non-intersecting non-empty subsequences
$\{b^j_1,...,b^j_{n_j}\}$ such that $b^j_1+\ldots +b^j_{n_j}=s_j$,
$\{a_1,\ldots ,a_v\}=\bigcup\limits_{j=1}^{m}\{b^j_1, \ldots ,b^j_{n_j}\}$,
is a {\it partition of $\{a_1,\ldots ,a_v\}$
of the type $\left(\begin{matrix} s_1\ldots s_m \\ n_1\ldots n_m\end{matrix}\right)$}. (The sequence
$(n_1, \ldots , n_m)$ is a partition of $v$ in the usual sense:
$\displaystyle{\sum_{i=1}^{m}n_i}=v$.)

\begin{lemma-definition} \label{ld2.3} The following relation holds for $k,\bar{k}>1$:
$$
\partial_{i_1}\dotsb \partial_{i_k}\bar\partial_{\bar i_1}\dotsb
\bar\partial_{\bar i_{\bar k}} F\Bigl\vert_{t_0}=
\sum\limits_{m=1}^\infty \!
\sum\limits_{{s_1+\ldots +s_m= i_1+\ldots +i_k=\bar{i}_1+\ldots +\bar{i}_{\bar{k}} \atop r_1+\dotsb +r_m=
k+\bar{k}-2}} \!\!\!
\tilde{N}_{\tiny{\left(\begin{matrix} i_1\ldots i_k \\ \bar{i}_1\ldots \bar{i}_{\bar{k}}\end{matrix}\right)}}\!\!
\left(\begin{matrix} s_1\ldots s_m \\ r_1\ldots r_m \end{matrix}\right) \partial_0^{r_1}f^{s_1}\ldots \partial_0^{r_m} f^{s_m},
$$
where
$$\tilde N_{\tiny{\left(\begin{matrix} i_1\ldots i_k \\ \bar{i}_1\ldots \bar{i}_{\bar{k}}\end{matrix}\right)}}
\left(\begin{matrix} s_1\ldots s_m \\ r_1\ldots r_m \end{matrix}\right)= \frac{i_1\ldots i_k\, \bar{i}_1\ldots \bar{i}_{\bar{k}}}{s_1\ldots s_m}
\sum T_{i_1\ldots i_k}\left(\begin{matrix} s_1\, \ldots \, s_m
\\ r_1\! -\! n_1+1
\ldots r_m\! -\! n_m+1\end{matrix}\right)$$
and the summation is carried over all partitions of the set
$\{\bar i_1,...,\bar i_{\bar k}\}$ of type
$\left(\begin{matrix} s_1\ldots s_m \\ n_1\ldots n_m\end{matrix}\right)$.
\end{lemma-definition}

\begin{proof}According to Lemma-Definition \ref{ld2.2}, $$\partial_{i_1}\ldots\partial_{i_k}\bar\partial_{\bar i_1}\ldots \bar\partial_{\bar i_{\bar k}}F$$
$$=\,\,
\bar\partial_{\bar i_{\bar i}}\ldots \bar\partial_{\bar i_{\bar k}}
\Bigl(\sum\limits_{m=1}^\infty \sum\limits_{{
s_1+\ldots+s_m=i_1+\ldots +i_k\atop l_1+\ldots +l_m=m+k-2}} \frac{i_1\ldots i_k}{s_1\ldots s_m}\cdot\ T_{i_1\ldots i_k}\left(\begin{matrix} s_1\ldots s_m \\ l_1\ldots l_m\end{matrix}\right)\partial^{l_1}_0\partial_{s_1} F\ldots \partial^{l_m}_0\partial_{s_m} F\Bigr )$$

$$=\sum\limits_{m=1}^\infty
\sum\limits_{{s_1+\ldots+s_m=i_1+\ldots +i_k\atop l_1+\ldots +l_m=m+k-2}}
\sum\left (\frac{i_1\ldots i_k}{s_1\ldots s_m} T_{i_1\ldots i_k}\left(\begin{matrix} s_1\ldots s_m \\ l_1\ldots
l_m\end{matrix}\right)\right.$$
$$\left.
\phantom{\frac{i}{s}}\times \,\,
\partial^{l_1}_0 \partial_{s_1}\bar\partial_{\bar j_1^1}
\ldots \bar\partial_{\bar j_{n_1}^1}F\ldots \partial^{l_m}_0 \partial_{s_m}\bar\partial_{\bar j_1^m}\ldots \bar
\partial_{\bar j_{n_m}^m} F\right ),$$
where the interior summation is carried over all partitions of
$\{\bar i_1,\ldots ,\bar i_{\bar k}\}$ of type $\left(\begin{matrix} s_1\ldots s_m \\ n_1\ldots n_m\end{matrix}\right)$.
Therefore, according to Lemma \ref{l2.2}, $$\partial_{i_1}\ldots\partial_{i_k}\bar\partial_{\bar i_1}\ldots  \bar\partial_{\bar i_{\bar k}}F$$
$$= \sum\limits_{m=1}^\infty
\sum\limits_{{s_1+\ldots+s_m=i_1+\ldots +
i_k\atop l_1+\ldots +l_m=m+k-2}}\!\!\!\!\!\!\!
\frac{i_1\ldots i_k\, \bar{i}_1\ldots
\bar{i}_{\bar{k}}}{s_1\dots s_m}\sum\left (
T_{i_1\ldots i_k}\left(\begin{matrix} s_1\ldots s_m \\ l_1\ldots
l_m\end{matrix}\right)\partial^{l_1+n_1-1}_0(f^{s_1})\dots
\partial^{l_m+n_m-1}_0(f^{s_m})\right )
$$
$$=\sum\limits_{m=1}^\infty
\sum\limits_{{s_1+\ldots +s_m=i_1+\ldots +i_k\atop
r_1+\ldots +r_m=k+\bar{k}-2}}\!\!\!\!
\frac{i_1\ldots i_k\, \bar{i}_1\ldots \bar{i}_{\bar{k}}}{s_1\dots s_m}
$$
$$
\times \, \sum\left
(T_{i_1\ldots i_k}\left(\begin{matrix} s_1\ldots s_m \\ r_1-n_1+1 \ldots r_m-n_m+1\end{matrix}\right)\partial^{r_1}_0(f^{s_1}) \dots\partial^{r_m}_0(f^{s_m})\right ),$$
where the interior summation is carried over all partitions of the set
$\{\bar i_1,\ldots ,\bar i_{\bar k}\}$ of type $\left(\begin{matrix} s_1\ldots s_m \\ n_1\ldots n_m\end{matrix}\right)$.
\end{proof}

\vspace{1ex}

For $k>1$ we put $$\tilde N_{\tiny{\left(\begin{matrix} i \\ i_1\ldots i_k\end{matrix}\right)}}
\left(\begin{matrix} s\\r \end{matrix}\right)=
\tilde N_{\tiny{\left(\begin{matrix} i_1\ldots i_k\\ i\end{matrix} \right)}}\left(\begin{matrix} s\\ r \end{matrix}\right)=\delta_{i,i_1+\ldots +i_k}\delta_{s,i}\delta_{r,k-1}i_1\ldots i_k$$
and
$$\tilde N_{\tiny{\left(\begin{matrix} i \\ i_1\ldots i_k\end{matrix}\right)}}
\left(\begin{matrix} s_1\ldots s_m \\ r_1\ldots r_m \end{matrix}\right)=\tilde N_{\tiny{\left(\begin{matrix} i_1\ldots i_k\\i\end{matrix}\right)}}
\left(\begin{matrix} s_1\ldots s_m \\ r_1\ldots r_m \end{matrix}\right)=0 \quad \text{for}\quad  m>1$$

Lemma \ref{l2.2} and Lemma-Definition \ref{ld2.3} implies

\begin{theorem}\label{t2.1}
Any symmetric formal solution to the dispersionless 2DTL has the form
\beq\label{series}
F=F(t_0)\,+ \sum\limits_{i>0}^{\infty}if^it_i\bar{t}_i  +\sum\limits_{|\Delta|=|\bar{\Delta}|}\!\! \sum
\limits_{{s_1+\ldots +s_m= |\Delta|\atop r_1+\ldots +r_m= \ell (\Delta)+\ell (\bar{\Delta})-2>0}} \!\!\!\!\!\!\!\!\!\!\!\!
N_{(\Delta|\bar{\Delta})} \left(\begin{matrix} s_1\ldots s_m \\ r_1\ldots r_m \end{matrix}\right) \partial_0^{r_1}(f^{s_1})\ldots
\partial_0^{r_m}(f^{s_m})\, t_{\Delta}\bar t_{\bar{\Delta}},
\eeq
where $f(t_0)=\exp (\p_0^2 F(t_0))$,
\beq\label{NN}
N_{(\Delta|\bar{\Delta})}
\left(\begin{matrix} s_1\ldots s_m \\ r_1\ldots r_m \end{matrix}\right) = \frac{1}{\sigma (\Delta)\sigma (\bar{\Delta})} \,
\displaystyle
\tilde N_{\tiny{\left(\begin{matrix} \mu_1\ldots \mu_k\\ \bar \mu_1 \ldots \bar \mu_{\bar k} \end{matrix}\right)}}
\left(\begin{matrix} s_1\ldots s_m \\ r_1\ldots r_m \end{matrix}\right)
\eeq
and $[\mu_1, \ldots , \mu_{\ell}]=\Delta$,
$[\bar \mu_1, \ldots , \bar \mu_{\ell}]=\bar{\Delta}$.
\end{theorem}

\vspace{1ex}

Theorem \ref{t2.1} implies that the symmetric solutions are parameterized by a function of one variable. The Taylor coefficients in the right hand side of equation (\ref{series}) are differential polynomials in the function $f(t_0)=e^{F''(t_0)}$ with universal coefficients $N_{(\Delta|\bar{\Delta})}\left(\begin{matrix} s_1\ldots s_m \\ r_1\ldots r_m \end{matrix}\right)$. Explicit formulas for the simplest
coefficients $N_{(\Delta|\bar{\Delta})}\left(\begin{matrix} s_1\ldots s_m \\ r_1\ldots r_m \end{matrix}\right)$ are given below.

\subsection{Explicit formulas for some coefficients}

\begin{lemma}\label{l3.1}
Let $\Delta$, $\bar{\Delta}$ be two Young diagrams such that
$|\Delta |=|\bar \Delta |=d$, $\ell(\Delta)+\ell(\bar{\Delta})>2 $. Then
$$N_{(\Delta|\bar{\Delta})} \left(\begin{matrix}|\Delta|\\\ell(\Delta)+\ell(\bar{\Delta})-2  \end{matrix}\right)=
\frac{\rho(\Delta) \rho(\bar{\Delta})} {d\sigma(\Delta)\sigma(\bar{\Delta})}$$
and $N_{(\Delta|\bar{\Delta})} \left(\begin{matrix}s\\r \end{matrix}\right)=0$
otherwise.
\end{lemma}

\begin{proof} According to our definitions
$T_{ij}\left(\begin{matrix} i+j\\ 1\end{matrix}\right)=
T_{ij}(i+j)=P_{ij}(i+j)=1$
and $T_{ij}\left(\begin{matrix} s \\ l\end{matrix}\right)=0$ in other cases.
Therefore, $$T_{i_1\ldots i_k}\left(\begin{matrix} i_1+\ldots+i_k\\ k-1\end{matrix}\right)=
T_{i_1\ldots i_{k-1}}\left(\begin{matrix} i_1+\ldots +i_{k-1} \\ k-2\end{matrix}\right)T_{(i_1+\ldots+i_{k-1})i_k}(i_1+\ldots+i_k)=1$$
and $T_{i_1\ldots i_k}\left(\begin{matrix} s \\ l\end{matrix}\right)=0$
in other cases.
Let the sequences $(i_1, \ldots , i_{\ell (\Delta )})$,
$(\bar i_1, \ldots , \bar i_{\ell (\bar \Delta)})$ be any permutations
of $[\mu_1, \ldots , \mu_{\ell}]$, $[\bar \mu_1, \ldots , \bar \mu_{\ell}]$
respectively, where
$[\mu_1, \ldots , \mu_{\ell}]= \Delta$, $[\bar \mu_1, \ldots , \bar \mu_{\ell}]=\bar{\Delta}$.
Then
$$\tilde N_{\tiny{\left(\begin{matrix} i_1\ldots i_k \\
\bar i_1\ldots \bar i_{\bar k} \end{matrix}\right)}} \left(\begin{matrix}|
\Delta|\\\ell(\Delta)+\ell(\bar{\Delta})-2  \end{matrix}\right)=\frac{\rho(\Delta) \rho(\bar{\Delta})} {|\Delta|}\,
T_{i_1\ldots i_k}\left(\begin{matrix}|\Delta|\\\ell(\Delta)-1 \end{matrix}\right)=\frac{\rho(\Delta) \rho(\bar{\Delta})} {|\Delta|}$$ and $\tilde N_{\tiny{\left(\begin{matrix} i_1\dots i_k \\ \bar i_1\dots \bar i_{\bar k} \end{matrix}\right)}} \left(\begin{matrix} s \\ r \end{matrix}\right)=0$ in other cases.
Hence $$N_{(\Delta|\bar{\Delta})} \left(\begin{matrix}|\Delta|\\\ell(\Delta)+\ell(\bar{\Delta})-2  \end{matrix}\right)= \frac{1}{\sigma(\Delta)\sigma(\bar{\Delta})} \, \tilde N_{\tiny{\left(\begin{matrix} i_1\ldots i_k \\ \bar i_1\dots \bar i_{\bar k} \end{matrix}\right)}} \left(\begin{matrix}|\Delta|\\\ell(\Delta)+\ell(\bar{\Delta})-2  \end{matrix}\right)=  \frac{\rho(\Delta) \rho(\bar{\Delta})} {|\Delta|\sigma(\Delta)\sigma(\bar{\Delta})}$$
and $N_{(\Delta|\bar{\Delta})} \left(\begin{matrix}s\\r \end{matrix}\right)=0$
otherwise.
\end{proof}

\begin{lemma}\label{l3.2}
$$\tilde N_{\tiny{\left(\begin{matrix} i_1 i_2 \\ \bar{i}_1 \bar{i}_{2}\end{matrix}\right)}}
\left(\begin{matrix} \bar{i}_1 \bar{i}_2 \\ 1 1 \end{matrix}\right)
= \tilde N_{\tiny{\left(\begin{matrix} i_1 i_2 \\ \bar{i}_1 \bar{i}_{2}\end{matrix}\right)}}
\left(\begin{matrix} \bar{i}_2 \bar{i}_1 \\ 1 1 \end{matrix}\right)
= -\frac{i_1i_2}{2\sigma([i_1, i_2])\sigma([\bar{i}_1, \bar{i}_2])}\min\{i_1, i_2, \bar{i}_1,\bar{i}_{2}\})$$
and $\tilde N_{\tiny{\left(\begin{matrix} i_1 i_2 \\ \bar{i}_1 \bar{i}_{2}\end{matrix}\right)}}
\left(\begin{matrix} s_1 s_2 \\ r_1 r_2 \end{matrix}\right) = 0$ otherwise.
\end{lemma}

\begin{proof}
The number $P_{ij}(r_1,r_2)$ is the number of solutions to
the system of equation $i_1+i_2=i$, $j_1+j_2=j$, $i_1+j_1=r_1$,
$i_2+j_2=r_2$ in positive integer numbers $i_1,i_2,j_1,j_2$.
This system is equivalent to the system $i_1=r_1-j_1$, $i_2=i+j_1-r_1$, $j_2=j-j_1$, $i+j=r_1+r_2$. Hence $P_{ij}(r_1,r_2)=0$ if $i+j\neq r_1+r_2$.

\vspace{1ex}

In other cases the number of solutions is the number of positive $j_1$ such that $r_1-j_1>0$, $i+j_1-r_1>0$, $j-j_1>0$ or $j_1<r_1$, $j_1<j$,  $j_1>r_1-i$. By symmetry in pairs $i,j$ and $r_1,r_2$ it is enough
to consider the case $j\leq i$ and $r_1\leq r_2$.
In this case the number of solutions is $\min\{r_1,j\}-1$. Thus $P_{ij}(r_1,r_2)=\min\{i,j,r_1,r_2\}-1$.

By our definition
$$T_{i_1, i_2}\left(\begin{matrix} s_1 s_2 \\ 1 1\end{matrix}\right)= T_{ij}(s_1,s_2)= -\frac{1}{2}\Bigl (P_{ij}(s_1+s_2)+ \frac{1}{2}P_{ij}(s_1,s_2)\Bigr )=-\frac{1}{2}\min\{i,j,s_1,s_2\}$$ if $i+j= s_1+s_2$ and $T_{i_1, i_2}\left(\begin{matrix} s_1 s_2 \\ l_1 l_2 \end{matrix}\right)=0$ in other cases.

By Lemma-Definition \ref{ld2.3} we have
$$\tilde N_{\tiny{\left(\begin{matrix}
i_1 i_2 \\ \bar{i}_1 \bar{i}_{2}\end{matrix}\right)}}
\left(\begin{matrix} s_1 s_2 \\ r_1 r_2 \end{matrix}\right) = \frac{i_1 i_2\, \bar{i}_1\bar{i}_{2}}{s_1s_2}
\sum T_{i_1i_2}\left(\begin{matrix} s_1 \ \ s_2
\\ r_1\! -\! n_1+1  \ \ r_2\! -\! n_2+1\end{matrix}\right),$$
where the summation is carried over all partitions
of the set $\{\bar{i}_1,\bar{i}_2\}$ of type
$\left(\begin{matrix} s_1 s_2 \\ n_1 n_2 \end{matrix}\right)$. Such partitions exist only if either
$\bar{i}_1=s_1$, $\bar{i}_2=s_2$ or $\bar{i}_1=s_2$, $\bar{i}_2=s_1$.
Therefore, $$\tilde N_{\tiny{\left(\begin{matrix} i_1 i_2 \\ \bar{i}_1 \bar{i}_{2}\end{matrix}\right)}}
\left(\begin{matrix} \bar{i}_1 \bar{i}_2 \\ 1 1 \end{matrix}\right) = \tilde N_{\tiny{\left(\begin{matrix} i_1 i_2 \\ \bar{i}_1 \bar{i}_{2}\end{matrix}\right)}}
\left(\begin{matrix} \bar{i}_2 \bar{i}_1 \\ 1 1 \end{matrix}\right) = \frac{i_1 i_2\, \bar{i}_1\bar{i}_{2}}{\bar{i}_1 \bar{i}_2}
T_{i_1i_2}\left(\begin{matrix} \bar{i}_1 \ \ \bar{i}_2
\\ 1  \ \ 1\end{matrix}\right)= -\frac{i_1i_2}{2}\min\{i_1, i_2, \bar{i}_1,\bar{i}_{2}\})$$
and $0$ in other cases.
Finally we obtain $$ N_{[ i_1 i_2]|[\bar{i}_1 \bar{i}_{2}]}
\left(\begin{matrix} \bar{i}_1 \bar{i}_2 \\ 1 1 \end{matrix}\right) = N_{[ i_1 i_2]|[\bar{i}_1 \bar{i}_{2}]}
\left(\begin{matrix} \bar{i}_2 \bar{i}_1 \\ 1 1 \end{matrix}\right) =-\frac{i_1i_2}{2\sigma([ i_1, i_2])\sigma([ \bar{i}_1, \bar{i}_2])}(\min\{i_1, i_2, \bar{i}_1,\bar{i}_{2}\})$$
and $N_{[ i_1 i_2]|[\bar{i}_1 \bar{i}_{2}]}
\left(\begin{matrix} s_1 s_2 \\ r_1 r_2 \end{matrix}\right) = 0$ otherwise.
\end{proof}

\vspace{2ex}

\section{Double Hurwitz numbers}

\vspace{1ex}

\subsection{Generating function for double Hurwitz numbers of genus 0}

Given a meromorphic function $\varphi :
\Omega \rightarrow \hat{\mathbb{C}}$ of degree $d$ on
a connected Riemann surface $\Omega$, one can associate with each point $p\in\hat{\mathbb{C}}=\mathbb{C}\cup \{\infty \}$
a Young diagram $\Delta =
\Delta(\varphi , p)=[\mu_1 , \ldots , \mu_{\ell}]$ such that $|\Delta |=
\deg \, \varphi =d$ and
$\mu_i$ equals the degree of the map $\varphi$
at the point $p^i$ of the
complete pre-image $\varphi ^{-1}(p)=\{p^1, \ldots, p^\ell\}$.
Informally speaking, $\mu_i$ is the number of sheets of the covering
that are glued together at the point $p^i$ lying above $p$. We say that meromorphic functions $\varphi :
\Omega \rightarrow \hat{\mathbb{C}}$ and $\varphi' :
\Omega' \rightarrow \hat{\mathbb{C}}$ are equivalent if there exists a biholomorphic map $f:\Omega\rightarrow\Omega'$ such that $\varphi=\varphi'f'$.

Fix different points $p_1, \ldots , p_k$ and Young diagrams
$\Delta_1,\ldots , \Delta_k$ such that $|\Delta_i|=d$
for all $i=1, \ldots , k$. Consider the sum
$$
H(\Delta_1,\dots,\Delta_k)=\sum_{\varphi \in
\Cov_d(\Omega, \{\Delta_1, \ldots,\Delta_k\})} \frac{1}{|\Aut(\varphi )|}\,,
$$
where $|\Aut(\varphi )|$ is the order of the automorphism group
the covering
$\varphi$, $\Cov_d(\Omega, \{\Delta_1, \ldots, \Delta_k\})$ is the set of biholomorphic equivalence classes of meromorphic functions
$\varphi :\Omega\rightarrow \hat{\mathbb{C}}$
such that $\Delta(\varphi ,p_i)=\Delta_i$
and $\Delta(\varphi ,p)=(1^d)$ at the other points.
This sum does not depend on positions of the points $p_i$. It is called the \textit{(connected) Hurwitz number}. See \cite{Lando} for a review.

\vspace{1ex}

The Hurwitz numbers $$H_{d,l}(\Delta|\bar{\Delta}) =
H(\Delta, \bar{\Delta},\underbrace{(1^{d-2}2^1), \ldots ,
(1^{d-2}2^1)}_{l})$$ for functions of degree $d$ with arbitrary
ramification types $\Delta$, $\bar{\Delta}$ at two
fixed points and $l$ simple critical values
of the type $(1^{d-2}2^1)$ are called {\it double Hurwitz numbers}.
In fact the notation $H_{d,l}(\Delta|\bar{\Delta})$ is redundant
because this number is zero unless $d=|\Delta |=|\bar \Delta |$.
The genus $g$ of $\Omega$ and the number $l$
are connected by the Riemann-Hurwitz relation
$2g-2=l-\ell (\Delta)- \ell (\bar{\Delta})$.
Introduce two infinite system of variables
${\bf t}=\{t_1, t_2, \ldots \}$, ${\bf \bar t}=\{\bar t_1, \bar t_2, \ldots \}$,
as in the previous section, and consider the generating function
\beq\label{gen}
\Phi (\beta , Q, {\bf t}, {\bf \bar t})=
\sum_{l\geq 0}\frac{\beta ^{l}}{l!}\sum_{d\geq 1}Q^d \! \sum_{|\Delta| = |\bar\Delta|=d}H_{d,l}(\Delta, \bar\Delta)\prod_{i=1}^{\ell (\Delta )}\mu_i t_{\mu_i}\!\prod_{i=1}^{\ell (\bar \Delta )}\bar \mu_i \bar t_{\bar \mu_i}\,,
\eeq
where $\Delta=[\mu_1,\ldots,\mu_{\ell (\Delta)}]$,
$\bar{\Delta}=[\bar{\mu}_1,\ldots,\bar{\mu}_{\ell (\bar{\Delta})}]$.

In order to extract the contribution of genus $g$ surfaces the
following trick can be used (see, e.g., \cite{BEMS11} for the
case of ordinary Hurwitz numbers).
Let us rescale $\{t_k\}$, $\{\bar t_k\}$ and $\beta$ by introducing a new parameter $\hbar$ as $t_k \to t_k/\hbar$, $\beta \to \hbar \beta$ and consider the modified generating function
$
\Phi (\hbar ;\beta , Q, {\bf t}, {\bf \bar t}):=
\hbar^2 \Phi (\hbar \beta , Q, {\bf t}/\hbar , {\bf \bar t}/\hbar )
$,
then the series (\ref{gen}) having regard to the Riemann-Hurwitz
formula acquires the form of the topological expansion
\beq\label{topol}
\Phi (\hbar ;\beta , Q, {\bf t}, {\bf \bar t})=
\sum_{g\geq 0}\hbar^{2g} \Phi_{g}(\beta , Q, {\bf t}, {\bf \bar t}),
\eeq
where
\beq\label{LG202}
\Phi_{g}=\sum_{d\geq 1}
\! \sum_{|\Delta |=|\bar \Delta |=d}
\frac{Q^d \, \beta^{\ell (\Delta )\! +\!
\ell (\bar \Delta )\! +\! 2g\! -\! 2}}{(\ell (\Delta )\! +\!
\ell (\bar \Delta )\! +\! 2g\! -\! 2)!}
H_{d,\ell (\Delta )\! +\!
\ell (\bar \Delta )\! +\!
2g\! -\! 2}(\Delta , \bar \Delta )
\prod_{i=1}^{\ell (\Delta )}\mu_i t_{\mu_i}\!
\prod_{i=1}^{\ell (\bar \Delta )}\bar \mu_i \bar t_{\bar \mu_i}
\eeq
counts the connected coverings of genus $g$.
In particular,
\beq\label{LG203}
\Phi_{0}=\sum_{d\geq 1}
\sum_{|\Delta |=|\bar \Delta |=d}
\frac{Q^d \, H_{d,\ell (\Delta )\! +\!
\ell (\bar \Delta )\! -\! 2}(\Delta ,
\bar \Delta )}{\beta^{2}(\ell (\Delta )\! +\!
\ell (\bar \Delta )\!  -\! 2)!}\,
\prod_{i=1}^{\ell (\Delta )}(\beta \mu_i t_{\mu_i})\!
\prod_{i=1}^{\ell (\bar \Delta )}(\beta \bar \mu_i \bar t_{\bar \mu_i})
\eeq
is the generating function for the numbers of
the ramified coverings $\hat{\mathbb{C}}\longrightarrow
\hat{\mathbb{C}}$.

\vspace{1ex}

In \cite{Okounkov00} Okounkov has proved that the
(dispersionfull) tau-function
$$\tau_n ({\bf t}, {\bf\bar  t})=e^{\frac{1}{12}\,\beta n
(n+1)(2n+1)}Q^{\frac{1}{2}\,n(n+1)}
\exp \Bigl (\Phi (\beta , e^{\beta (n+\frac{1}{2})}
Q, {\bf t}, {\bf \bar t})\Bigr )
$$
solves the 2DTL hierarchy of Ueno and Takasaki \cite{UenoTakasaki}.
The dispersionless limit was characterized in terms of
string equations by Takasaki \cite{Takasaki12}.
The procedure of passing to the dispersionless limit \cite{TakTak}
is equivalent to extracting the contribution of genus $0$ surfaces.
Thus the function
\beq\label{FF}
\begin{array}{c}\displaystyle{
F(\beta , Q, t_0, {\bf t},{\bf \bar{t}})=
\frac{\beta t_0^3}{6} +\frac{t_0^2}{2} \log Q +\Phi_0
(\beta , Q e^{\beta t_0}, {\bf t}, {\bf \bar t}) }
\\ \\
\displaystyle{=\,
\frac{\beta t_0^3}{6} +\frac{t_0^2}{2} \log Q + \sum_{d\geq 1}
\sum_{|\Delta |=|\bar \Delta |=d}
\frac{(Q e^{\beta t_0})^d \, H_{d,\ell (\Delta )\! +\!
\ell (\bar \Delta )\! -\! 2}(\Delta , \bar \Delta )}{\beta^{2}(\ell (\Delta )\! +\!
\ell (\bar \Delta )\!  -\! 2)!}\,
\prod_{i=1}^{\ell (\Delta )}(\beta \mu_i t_{\mu_i})\!
\prod_{i=1}^{\ell (\bar \Delta )}(\beta \bar \mu_i \bar t_{\bar \mu_i})}
\end{array}
\eeq
satisfies the dispersionless 2DTL hierarchy from the previous section.
Moreover, from the structure of the series we see that it is a symmetric
solution.

\vspace{1ex}

\subsection{Formulas for double Hurwitz numbers}\label{ss3.2}

Now we can use Theorem \ref{t2.1} to calculate the genus 0 double Hurwitz numbers. Note that at genus 0 the Hurwitz number
$H_{d,l}(\Delta|\bar{\Delta})$ depends on the diagrams $\Delta , \bar \Delta$
only (the numbers $d,l$ are restored as $d=|\Delta |=|\bar \Delta |$,
$l= \ell (\Delta )+\ell (\bar \Delta )-2$). We will write
$H_0(\Delta|\bar{\Delta})=H_{d,l}(\Delta|\bar{\Delta})$.
It is obviously that $H_0([n]|[n'])=\frac{\delta_{n,n'}}{n}$

\begin{theorem}\label{t3.1} The genus 0 double Hurwitz number $H_0(\Delta|\bar{\Delta})$ for $\ell (\Delta)+
\ell (\bar{\Delta})\! > 2$ is given by
$$\label{hurwitz0}H_0(\Delta|\bar{\Delta})= \frac{(\ell (\Delta)+
\ell (\bar{\Delta})\! -\! 2)!} {\rho (\Delta)\, \rho (\bar{\Delta})}
 \sum s_1^{r_1}\ldots s_m^{r_m} \, N_{(\Delta|\bar{\Delta})}
\left(\begin{matrix} s_1 \ldots s_m \\ r_1\ldots r_m \end{matrix}\right)$$
where the sum is carried over all matrices
$\left(\begin{matrix} s_1 \ldots s_m \\ r_1\ldots r_m \end{matrix}\right)$ such that $s_1\! + \ldots +\! s_m= |\Delta|$ and $r_1\! +\! \ldots + \!r_m \! =\! \ell (\Delta)\! +\! \ell (\bar{\Delta})\! -\!2$
\end{theorem}

\begin{proof} The function
$$F(1,1, t_0,{\bf t},{\bf \bar{t}})=
\frac{t_0^3}{6} +\sum_{|\Delta |=|\bar \Delta|}
\frac{e^{|\Delta | t_0} \,
H_0(\Delta , \bar \Delta )}{(\ell (\Delta)\! +
\!\ell (\bar \Delta )\!  -\! 2)!}\,\prod_{i=1}^{\ell (\Delta )} (\mu_i t_{\mu_i})\!\prod_{i=1}^{\ell (\bar \Delta )}(\bar \mu_i \bar t_{\bar \mu_i})$$ is a symmetric solution of the
dispersionless 2DTL hierarchy with $f(t_0)=e^{t_0}$.
Therefore, according to Theorem \ref{t2.1},
$$F(1,1, t_0, {\bf t},{\bf \bar{t}})=
\frac{t_0^3}{6}+\sum\limits_{|\Delta|=|\bar{\Delta}|}\!\!\! \!\!\!  \sum
\limits_{\tiny{\begin{matrix}  s_1+\ldots +s_m= |\Delta| \\
r_1\! +\! \ldots \! +\! r_m \!=\! \ell (\Delta)\! +\! \ell (\bar{\Delta})\! -\! 2\end{matrix}}}
\!\!\! \!\!\! \!\!\! \!\!\! \!\!\! \!\!\!
N_{(\Delta|\bar{\Delta})} \left(\begin{matrix} s_1\ldots s_m \\ r_1\ldots r_m \end{matrix}\right) \partial_0^{r_1}(e^{s_1 t_0})\ldots
\partial_0^{r_m}(e^{s_m t_0})\, t_{\Delta}\bar{t}_{\bar{\Delta}}
$$
Comparing this with the previous formula, we find that
$$\rho (\Delta)\rho (\bar{\Delta})
\frac{H_0(\Delta ,
\bar \Delta )}{(\ell (\Delta)\! +\!\ell
(\bar \Delta )\!  -\! 2)!} \,\, =\!\!\!\!\!\!
\sum\limits_{\tiny{\begin{matrix} s_1+\ldots +s_m= |\Delta|
\\ r_1+\ldots +r_m=\ell (\Delta)+\ell (\bar{\Delta})-2
\end{matrix}}} \!\!\! \!\!\!
N_{(\Delta|\bar{\Delta})} \left(\begin{matrix} s_1\ldots s_m \\ r_1\ldots r_m \end{matrix}\right) s_1^{r_1}\ldots s_m^{r_m}$$
which is the assertion of the theorem.
\end{proof}

{\bf Remark.}
Theorem \ref{t3.1} and the geometric definition of Hurwitz numbers give some nontrivial combinatorial relations. First, the sum
$\sum s_1^{r_1}\ldots s_m^{r_m} \, N_{(\Delta|\bar{\Delta})}
\left(\begin{matrix} s_1 \ldots s_m \\ r_1\ldots r_m \end{matrix}\right)$
must be positive despite the fact that the numbers $N_{(\Delta|\bar{\Delta})}\left(\begin{matrix} s_1 \ldots s_m \\ r_1\ldots r_m \end{matrix}\right)$ may have different signs.
Second, the equality $H_0(\Delta|\bar{\Delta})=H_0(\bar{\Delta}|\Delta)$
implies the identity
$$\sum s_1^{r_1}\ldots s_m^{r_m} \, N_{(\Delta|\bar{\Delta})} \left(\begin{matrix} s_1 \ldots s_m \\ r_1\ldots r_m \end{matrix}\right)=\sum s_1^{r_1}\ldots s_m^{r_m} \, N_{(\bar{\Delta}|\Delta)} \left(\begin{matrix} s_1 \ldots s_m \\ r_1\ldots r_m \end{matrix}\right)$$
which is non-trivial because all the definitions
depend on the order of the Young diagrams.

\vspace{1ex}

It follows from Theorem \ref{t3.1} that the formulas for  $N_{(\Delta|\bar{\Delta})}\left(\begin{matrix} s_1 \ldots s_m \\ r_1\ldots r_m \end{matrix}\right)$ give explicit expressions for the genus 0 double Hurwitz numbers. Let us consider some examples.

\vspace{1ex}

\begin{corollary}\label{t3.2} The number of polynomials of degree $n$ with a single  critical value at 0 of the type $\Delta$ is $$H_0(\Delta|[n]) = \frac{(\ell (\Delta)-1)!} {\sigma (\Delta)} \, n^{\ell (\Delta)-2}+\frac{\delta_{\ell (\Delta),1}}{n}.$$
\end{corollary}

\begin{proof} It follows from our definition that  $N_{(\Delta|[n])} \left(\begin{matrix} r_1\ldots r_m \\ n_1\ldots n_m \end{matrix} \right)=0$ for $m>1$. Hence, according to Theorem \ref{t3.1} and Lemma \ref{l3.1},
$$
H_0(\Delta|[n])=
\frac{(\ell (\Delta)-1)!}{\rho (\Delta)\rho ([n])} \, |\Delta|^{\ell (\Delta)-1}\frac{\rho (\Delta)
\rho ([n])}{|\Delta|\sigma (\Delta)}=\frac{(\ell (\Delta)-1)!} {\sigma (\Delta)} \, n^{\ell (\Delta)-2}.
$$
\end{proof}

\vspace{1ex}

\begin{corollary}\label{t3.3} Let $i_1+i_2=\bar{i}_1+\bar{i}_2=d$. Then
$$H_0([i_1,i_2]|[\bar{i}_1,\bar{i}_2])=2\, \frac{d
- \min\{i_1, i_2, \bar{i}_1,\bar{i}_{2}\}}{(1+\delta_{i_1i_2}) (1+\delta_{\bar{i}_1\bar{i}_2})}$$
\end{corollary}

\begin{proof} Put $\Delta=[i_1,i_2]$, $\bar{\Delta}= [\bar{i}_1,\bar{i}_2]$.
It follows from our definition that  $N_{(\Delta|\bar{\Delta})} \left(\begin{matrix} r_1\ldots r_m \\ n_1\ldots n_m \end{matrix} \right)=0$ for $m>2$. According to Theorem \ref{t3.1} and Lemmas \ref{l3.1}, \ref{l3.2}
we thus have:
$$H_0(\Delta|\bar{\Delta})= \frac{(\ell (\Delta)+
\ell (\bar{\Delta})\! -\! 2)!} {\rho (\Delta)\, \rho (\bar{\Delta})}
 \sum s_1^{r_1}\ldots s_m^{r_m} \, N_{(\Delta|\bar{\Delta})}
\left(\begin{matrix} s_1 \ldots s_m \\ r_1\ldots r_m \end{matrix}\right)$$

$$=\frac{(\ell (\Delta)+\ell (\bar{\Delta})\! -\! 2)!} {\rho (\Delta)\, \rho (\bar{\Delta})}
\left( |\Delta|^{\ell(\Delta)+\ell(\bar{\Delta})-2 } N_{(\Delta|\bar{\Delta})}\left(\begin{matrix} |\Delta| \\ \ell(\Delta)+
\ell(\bar{\Delta})-2 \end{matrix}\right)+ 2\bar{i}_1\bar{i}_2\, \tilde N_{\tiny{\left(\begin{matrix} i_1 i_2 \\ \bar{i}_1 \bar{i}_{2}\end{matrix}\right)}}
\left(\begin{matrix} \bar{i}_1 \bar{i}_2 \\ 1 1 \end{matrix}\right)\right) $$
$$=\frac{(\ell (\Delta)+\ell (\bar{\Delta})\! -\! 2)!} {\rho (\Delta)\, \rho (\bar{\Delta})}
\left( |\Delta|^{\ell(\Delta)+\ell(\bar{\Delta})-2 }
\frac{\rho(\Delta) \rho(\bar{\Delta})} {|\Delta|\sigma(\Delta)\sigma(\bar{\Delta})}- \frac{i_1i_2\bar{i}_1\bar{i}_2}{\sigma(\Delta)\sigma(\bar{\Delta})} \min\{i_1, i_2, \bar{i}_1,\bar{i}_{2}\}\right)  $$
$$=\frac{(\ell (\Delta)+\ell (\bar{\Delta})\! -\! 2)!} {\sigma(\Delta)\sigma(\bar{\Delta})}
\left( |\Delta|^{\ell(\Delta)+\ell(\bar{\Delta})-3 }
- \min\{i_1, i_2, \bar{i}_1,\bar{i}_{2}\}\right) =2\,
\frac{|\Delta|
- \min\{i_1, i_2, \bar{i}_1,\bar{i}_{2}\}} {\sigma(\Delta)
\sigma(\bar{\Delta})}$$
\end{proof}

\vspace{2ex}

\section{Conformal dynamics}

\vspace{1ex}

The result of section 2 takes on a geometric significance when
the Toda times $t_k$, $\bar t_k$ are identified with (complex conjugate)
moments of simply-connected
domains in the complex plane with smooth boundary.
In this case the dispersionless
2DTL dynamics encodes the shape dependence
of the conformal mapping of such a domain to some
fixed reference domain, thus providing an
effectivization of the Riemann mapping theorem.
We call it the {\it conformal dynamics}.

\subsection{Local coordinates in the space of simply-connected domains}

Let ${\sf D}\subset \mathbb{C}$ be a compact simply-connected domain whose boundary is a smooth curve $\gamma=\partial {\sf D}$ and let
${\sf D^c}=\hat{\mathbb{C}} \setminus {\sf D}$ be its complement in the
Riemann sphere $\hat{\mathbb{C}}$.
Without loss of generality, it is convenient to
assume that ${\sf D}\ni 0$ (and so ${\sf D^c}\ni \infty$).

\vspace{1ex}

Fix a real-analytic and real-valued function $U(z, \bar z)$
in $\mathbb{C}^*=\mathbb{C}\setminus \{0\}$ such that
$$\sigma (z, \bar z):=\partial\bar\partial U(z, \bar z)>0.$$
We introduce the set of moments of the
domain ${\sf D^c}$ as follows:
$$t_k=\frac{1}{2\pi ik}\oint_{\gamma} z^{-k}\partial U (z, \bar z) \, dz =-\, \frac{1}{\pi k}\int \!\!\! \int_{{\sf D^c}}\! z^{-k}
\sigma  (z, \bar z) \, d^2 z  \,,\quad k\geq 1,
$$
where $d^2z \equiv dx \, dy$. In general they are complex numbers.
We claim that together with the real parameter
$$t_0 = \frac{1}{2\pi i}\oint_{\gamma}\partial U (z, \bar z) \, dz
=\frac{1}{\pi }\int \!\!\! \int_{{\sf D}}\!
\sigma  (z, \bar z) \,  d^2 z
$$
they form a set of local coordinates in the space $\mathcal{H}$ of such domains ${\sf D}$ (or ${\sf D^c}$).
The function $\sigma$ plays the role of a
background charge density in the complex plane.

\vspace{1ex}

First we prove the local uniqueness of a domain with
given moments \cite{EV,KMZ05}.

\begin{proposition}\label{local-uniqueness}
Any one-parameter deformation ${\sf D}(t)$
of ${\sf D}={\sf D}(0)$ with some real parameter $t$
such that all $t_k$ are preserved, $\p_t t_k=0$, $k\geq 0$, is trivial.
\end{proposition}

\begin{proof}
The proof is a modification
of the one presented in \cite{KMZ05} for the case $\sigma (z, \bar z)=1$.
It is based on the following two facts.
\begin{itemize}
\item
{\it The difference of the boundary values $\p_t C^{\pm}(\zeta )d\zeta$
of the $t$-derivative of the Cauchy integral
$$
C(z)dz = \frac{dz}{2\pi i}\oint_{\gamma}
\frac{U_{\zeta}(\zeta , \bar \zeta )\, d\zeta}{\zeta -z}\,,
\quad \quad U_{\zeta}(\zeta , \bar \zeta )\equiv
\p_{\zeta}U (\zeta , \bar \zeta ),
$$
is a purely imaginary differential on $\gamma$.}
Indeed, let $\zeta (\theta , t)$ be a parametrization of the curve
$\gamma (t)$. Taking the $t$-derivative of the Cauchy integral
and integrating by parts, one gets
$$
\begin{array}{lll}
\p_t C(z)dz &=&\displaystyle{\frac{dz}{2\pi i}\oint_{\gamma} \left (
\frac{U_{\zeta \zeta}\zeta_t \bar \zeta_{\theta}+
U_{\zeta \bar \zeta}\bar \zeta_t \zeta_{\theta}+
U_{\zeta}\zeta_{t\theta}}{\zeta \, - \, \theta}-
\frac{U_{\zeta}\zeta_t \zeta_{\theta}}{(\zeta -z)^2}\right )
d\theta}
\\ &&\\
&=&\displaystyle{\frac{dz}{2\pi i}\oint_{\gamma}
\frac{\bar \zeta_t \zeta_{\theta} -\zeta_t \bar \zeta_{\theta}}{\zeta -z}
\, \sigma \bigl (\zeta (\theta , t), \bar \zeta (\theta , t)\bigr )d\theta}
\end{array}
$$
Hence, $\Bigl (\p_t C^+(\zeta )-\p_t C^-(\zeta )\Bigr )d\zeta =
\sigma (\zeta , \bar \zeta )(\p_t \bar \zeta d\zeta - \p_t \zeta
d\bar \zeta )=2i \sigma \, {\rm Im} \Bigl ( \p_t \bar \zeta d\zeta \Bigr )$
is indeed purely imaginary.
\item
{\it If a $t$-deformation preserves all the moments $t_k$,
$k\geq 0$, then the differential $\sigma (\zeta , \bar \zeta )
(\p_t \bar \zeta d\zeta - \p_t \zeta
d\bar \zeta )$ extends to a holomorphic differential in
${\sf D^c}$.} Take a small neighbourhood ${\sf B}$
of $0\in {\sf D}$ such that
$|z|<|\zeta |$ for all $z\in {\sf B}$ and $\zeta \in \gamma$, then
for $z\in {\sf B}$ we can expand:
$$
\p_t C^{+}(z)dz=\frac{\p}{\p t}\left (
\frac{dz}{2\pi i}\sum_{k\geq 0} z^k \oint_{\gamma}
\zeta^{-k-1}U_{\zeta}(\zeta , \bar \zeta )\, d\zeta \right )=
\sum_{k\geq 1}k (\p_t t_k) z^{k-1}dz=0
$$
and, since $C^+$ is analytic in ${\sf D}$, we conclude that
$\p_t C^+ \equiv 0$. Therefore, the differential
$\sigma (\zeta , \bar \zeta )
(\p_t \bar \zeta d\zeta - \p_t \zeta
d\bar \zeta )$ is the boundary value of the differential
$-\p_t C^-(z)dz$ which has at most simple pole at $\infty$ and
holomorphic everywhere else in ${\sf D^c}$. The equality
$$
\p_t t_0=\frac{1}{2\pi i}\oint_{\gamma}
\sigma (\zeta , \bar \zeta )
(\p_t \bar \zeta d\zeta - \p_t \zeta
d\bar \zeta )=0
$$
then implies that the residue at $\infty$ vanishes, so
$\p_t C^-(z)dz$ is a holomorphic differential.
\end{itemize}
Any holomorphic differential which is purely imaginary along the
boundary of a simply-connected domain must identically vanish in this domain.
Hence we conclude that $\p_t \bar \zeta d\zeta - \p_t \zeta
d\bar \zeta =0$ which means that the vector $\p_t \zeta$ is tangent
to the boundary and hence the deformation is trivial.
\end{proof}

In fact the assertion of the Proposition remains true under
less restrictive conditions on the function $\sigma$: it is enough
to require that $\sigma \neq 0$ in some strip-like neighbourhood
of the contour $\gamma$.

\vspace{1ex}

The fact that the set of the moments is not overcomplete follows
from the explicit construction of vector fields in the space of
domains that change real or imaginary part of any moment keeping
all the others fixed (see below). These arguments allow one to prove
the following theorem.

\begin{theorem}\label{local-coordinates}
The real parameters $t_0$, ${\rm Re}\, t_k$, ${\rm Im}\, t_k$,
$k\geq 1$, form a set of local coordinates in the space of simply-connected
plane domains with smooth boundary.
\end{theorem}

This statement allows one to identify functionals on the space
of domains ${\sf D}$ with functions of infinitely many independent
variables $t_0 , \{ t_k \}, \{ \bar t_k \}$.

\subsection{The Green function of the Dirichlet boundary value problem and
special deformations}
According to the Riemann mapping theorem,
there exists a conformal bijection $w: {\sf D^c}\rightarrow {\sf U}$
between ${\sf D^c}$ and
the exterior of the unit disk ${\sf U}=\{u\in\hat{\mathbb{C}}| \,\, |u|>1\}$.
If the conformal map $w(z)$ is known,
one can construct the Green function of the domain ${\sf D^c}$,
\begin{equation} \label{Green}
G(z,\xi)=
\frac{1}{2\pi}\log\left\vert\frac{w(z)-w(\xi)}{w(z)
\overline{w(\xi)}-1}\right\vert ,
\end{equation}
which solves the Dirichlet  boundary value problem
in ${\sf D^c}$. To wit, the Poisson formula
\beq\label{Poisson}
u(z)=-\oint_{\gamma}u_0(\xi)\partial_{n_{\xi}}
G(z,\xi)|d\xi|
\eeq
restores a
harmonic function in ${\sf D^c}$
from its boundary value $u_0=u|_{\gamma}$.
Here $\partial_{n_{\xi}}$ denotes the derivative along the
outward normal vector to the boundary of ${\sf D}$
with respect to the second variable and
$|d\xi|$ is an infinitesimal element of length along the boundary.
The Green function $G(z,\xi)=
\frac{1}{2\pi}\log|z-\xi|+g(z,\xi)$ in
${\sf D^c} \times {\sf D^c}$ is uniquely
defined by the following properties \cite{HC}:

\begin{itemize}
\item $G(z,\xi)=G(\xi,z)$
and $G(z,\xi')=0$
for any $z\in {\sf D^c}$ and $\xi'\in \gamma$;
\item The function $g(z,\xi)$ is continuous in ${\sf D^c}\times
{\sf D^c}$ and is continuous for $\xi$ in the closure of
${\sf D^c}$ at any $z\in {\sf D^c}$;
\item The function $g(z,\xi)$
is harmonic in $z$ for any $\xi\in {\sf D^c}$ and is harmonic in $\xi$ for any $z\in {\sf D^c}$.
\end{itemize}

\vspace{1ex}

Let us introduce the differential operator
\beq\label{nabla}
\nabla(z) =\partial_0 + D(z)+ \bar D(\bar z),
\eeq
where $D(z), \bar D(\bar z)$ are given by (\ref{DD}).

\vspace{1ex}

Fix a point $a\in {\sf D^c}$ and consider a special
infinitesimal deformation
of the form
\beq\label{special}
\delta_a n (z)=-\frac{\varepsilon \pi}{\sigma (z, \bar z)}\,
\p_{n_z} G(a,z)\,, \quad \quad z\in \gamma, \,\,\, \varepsilon \to 0,
\eeq
where $\delta_a n(z)$ is the normal displacement of the
boundary (positive if directed
outward ${\sf D}$) at the boundary point (equivalently,
one may speak about the normal ``velocity''
of the boundary deformation
which is $V_a(z)=\mbox{lim}_{\varepsilon \to 0}
(\delta n_a(z)/\varepsilon )$. For any sufficiently smooth initial
boundary this deformation is well-defined as $\varepsilon \to 0$.
By $\delta_a$ we denote the variation of any quantity under this deformation.

\begin{lemma}\label{change}  Let $X$ be any functional on the space
of domains ${\sf D}$ regarded as a function of $t_0, \{ t_k\},
\{\bar t_k\}$, then
for any $z\in {\sf D^c}$ we have $\delta_z X =\varepsilon \nabla(z)X$.
\end{lemma}

\begin{proof} It is easy to see that
$$
\delta_z t_0 =-\varepsilon
\oint_{\gamma} \p_{n_{\xi}} G(z, \xi ) |d\xi |=\varepsilon\,,\quad
\delta_z t_k =-\frac{\varepsilon}{k}
\oint_{\gamma} \xi^{-k}\p_{n_{\xi}} G(z, \xi ) |d\xi |=
\frac{\varepsilon}{k}\, z^{-k}
$$
by virtue of the Poisson formula (\ref{Poisson}).
Therefore by Theorem \ref{local-coordinates} we have:
$$
\delta_z X=\frac{\partial X}{\partial t_0}\, \delta_z t_0 +
\sum_{k\geq 1}\frac{\partial X}{\partial t_k}\, \delta_z t_k +
\sum_{k\geq 1}\frac{\partial X}{\partial \bar t_k}\, \delta_z \bar t_k=
\varepsilon
\Bigl (\partial_0 + \sum_{k\geq 1}\frac{z^{-k}}{k}\, \partial_k+
\sum_{k\geq 1} \frac{\bar{z}^{-k}}{k}\, \bar{\partial}_k\Bigr )X.
$$
\end{proof}

Now we can explicitly define the deformations that change only
either $x_k={\rm Re}\, t_k$ or $y_k={\rm Im}\, t_k$ keeping all
other moments fixed. From the proof of Lemma \ref{change} it follows that
the normal displacements $\delta n(\xi )=\varepsilon
{\rm Re}\, (\p_{n_{\xi}} H_k(\xi ))$ and
$\delta n(\xi )=\varepsilon
{\rm Im}\, (\p_{n_{\xi}} H_k(\xi ))$, where
$$
H_k(\xi )=-i\oint_{\infty}z^k \p_z G(z, \xi )\, dz
$$
(the contour integral goes around infinity) change the real and
imaginary parts of $t_k$ by $\pm \varepsilon$ respectively
keeping all
other moments unchanged.
In particular, the deformation
$$\delta_{\infty}n(\xi )=-
\frac{\varepsilon \pi}{\sigma (\xi , \bar \xi )}\, \p_{n_{\xi}}
G(\infty , \xi )
$$
changes $t_0$ only.
Therefore, the vector fields
$\p/ \p t_0$, $\p /\p x_k$, $\p / \p y_k$ in the space of
domains are locally well-defined and commute. Existence of such vector
fields means that the variables $t_k$ are independent and
$\p_k = \frac{1}{2}(\p_{x_k}-i\p_{y_k})$,
$\bar \p_k = \frac{1}{2}(\p_{x_k}+i\p_{y_k})$ can be understood
as partial derivatives.

\vspace{1ex}

We will also need the following simple lemmas.

\begin{lemma}\label{harmonic}
Let $X$ be a functional of the form
$X=\iint_{\sf D} \Psi (\zeta , \bar \zeta )
\, \sigma (\zeta , \bar \zeta ) \, d^2\zeta$ with an
arbitrary domain-independent integrable
function $\Psi$ regular on the boundary, then
$$
\nabla (z)X= \pi \Psi ^H (z),
$$
where $\Psi^H(z)$ is the (unique) harmonic extension of the function
$\Psi$ from the boundary to the domain ${\sf D^c}$.
\end{lemma}

\begin{proof}
The variation of $X$ under the special deformation
(\ref{special}) is
$$\delta_z X=\oint_{\gamma}\Psi (\zeta , \bar \zeta )
\sigma (\zeta , \bar \zeta )\delta n_z (\zeta )\, |d\zeta |=
-\varepsilon \pi \oint_{\gamma}\Psi (\zeta , \bar \zeta )
\p_{n_{\zeta}}G(z, \zeta )\, |d\zeta |
$$
Now the assertion obviously follows from Lemma \ref{change} and
the Poisson formula (\ref{Poisson}).
\end{proof}

\begin{lemma}\label{harmonic1}
Let $X$ be a functional of the form
$$X=\iint_{\sf D} \iint_{\sf D}  \sigma (\zeta _1, \bar \zeta_1 )
\Psi (\zeta _1, \bar \zeta _1;
\zeta _2, \bar \zeta _2)
\, \sigma (\zeta _2, \bar \zeta_2 )\,
d^2\zeta_1 d^2\zeta_2$$ with an
arbitrary domain-independent integrable
function $\Psi$ regular on the boundary, then
$$
\nabla (z)X= 2\pi \Phi^H (z),
$$
where $\Phi (z)=\iint_{\sf D}\Psi (z, \bar z;
\zeta , \bar \zeta )\sigma (\zeta , \bar \zeta )\,
d^2\zeta$.
\end{lemma}

\noindent
The proof is similar to that of Lemma \ref{harmonic}.

\subsection{The dispersionless tau-function}
Consider the following functional on the space of domains ${\sf D}$:
\beq\label{F}
F=-\, \frac{1}{\pi^2}
\iint_{\sf D} \iint_{\sf D}
\sigma (z, \bar z)\log \left |z^{-1} - \zeta^{-1}\right |
\sigma (\zeta , \bar \zeta )\, d^2 z d^2\zeta .
\eeq

\begin{theorem}\label{l4.2}(\cite{Takh,KKMWZ,MWZ,Ztmf}) It holds
\beq\label{GF}
G(z, \zeta )=\frac{1}{2\pi}\log \, \bigl | z^{-1} - \zeta ^{-1}\bigr |
+\frac{1}{4\pi} \nabla (z) \nabla (\zeta )F.
\eeq
\end{theorem}

\begin{proof}
The proof consists in successive application of Lemmas \ref{harmonic1},
\ref{harmonic}
and using the characteristic properties of the Green function.
Applying Lemma \ref{harmonic1} to (\ref{F}), we get:
\beq\label{nablaF}
\nabla (z)F= -\frac{2}{\pi}\iint_{\sf D} \log
\left |z^{-1}-\zeta^{-1}\right |\sigma (\zeta , \bar \zeta )\, d^2\zeta
\eeq
for $z\in {\sf D^c}$. Applying Lemma \ref{harmonic}, we conclude that
$\nabla (\zeta )\nabla (z)F$ is the harmonic continuation of the
function $-2\log \left |z^{-1}-\zeta^{-1}\right |$ from the boundary
to the domain ${\sf D^c}$. This function is harmonic everywhere in
${\sf D^c}$ except at $\zeta =z$, where it has the logarithmic
singularity. It can be cancelled, without changing
the boundary value, by adding the function $4\pi G(z, \zeta )$.
\end{proof}

Let $w(z)$ be the conformal map from ${\sf D^c}$ onto the
exterior of the unit circle normalized by the conditions
$w(\infty )=\infty$ and $w'(\infty )$ is real positive. It has the form
$\displaystyle{
w(z)=pz+\sum\limits_{j\geq 0} p_j z^{-j}}
$,
where $p>0$.

\begin{corollary} \label{t4.1} The conformal map $w(z)$ is given by
\beq\label{w(z)}
w(z)=z\exp
\left (\Bigl (-\frac 12 \partial_0^2-\partial_0 D(z)\Bigr )F\right ).
\eeq
\end{corollary}

\begin{proof} From equation (\ref{Green}) it follows that
$2\pi G(z, \infty )=-\log |w(z)|$. Tending $\xi \to \infty$ in
(\ref{GF}) and separating holomorphic and antiholomorphic parts
in $z$, we get the result.
\end{proof}

\noindent
Note that the limit $z\rightarrow\infty$  in (\ref{w(z)})
yields
$\log p=-\frac{1}{2}\, \partial_0^2F.$

\vspace{1ex}

The following theorem establishes the embedding of the conformal
dynamics into the dispersionless 2DTL hierarchy and identifies
$F$ with the dispersionless tau-function.

\begin{theorem} (\cite{KMZ05,Ztmf})\label{t4.2}
The function $F$ satisfies the equations
$$ \label{2}(z-\xi)e^{D(z)D(\xi)F}=z e^{-\partial_0D(z)F}-\xi e^{-\partial_0D(\xi)F},$$
$$ \label{3}(\bar z-\bar\xi)e^{\bar D(\bar z)\bar D(\bar\xi)F}=
\bar z e^{-\partial_0\bar D(\bar z)F}-\bar\xi
e^{-\partial_0\bar D(\bar\xi)F},$$
$$ \label{4}1- e^{-D(z)\bar D(\bar\xi)F}=\frac{1}{z\bar\xi} \,
e^{\partial_0 (\partial_0+D(z)+\bar D(\bar\xi))F}
$$
(cf. (\ref{eq1})-(\ref{eq3})).
\end{theorem}

\begin{proof}
The proof is the same as in \cite{KMZ05}.
Combining (\ref{Green}) and (\ref{GF}), we get
$$\log\left\vert\frac{w(z)-w(\xi)}{1-w(z)\bar w(\xi)}
\right\vert= \log\left\vert\frac{1}{z} - \frac{1}{\xi}\right\vert + \frac12\nabla(z)\nabla(\xi)F.
$$
Next, substituting here $w(z)$ from (\ref{w(z)}) and separating
holomorphic and antiholomorphic parts
in $z$, $\xi$, we obtain the desired result.
\end{proof}

{\bf Important remark.} Note that although the definitions
of the moments and the function $F$ essentially depend on the background
density $\sigma$, the formulas for the Green function and the
conformal map (\ref{GF}), (\ref{w(z)}) are $\sigma$-independent.
This means that the conformal dynamics is described by any solution
to the dispersionless 2DTL hierarchy of this class.

\subsection{Complimentary moments}
The integral of Cauchy type
$$
C(z)=\frac{1}{2\pi i}\oint_{\gamma}
\frac{U_{\zeta}(\zeta , \bar \zeta )\, d\zeta}{\zeta -z}\,,
\quad \quad
U_{\zeta}(\zeta , \bar \zeta )\equiv \p_{\zeta}U(\zeta , \bar \zeta ),
$$
defines a function which is analytic in ${\sf D}$ and ${\sf D^c}$
with a jump across $\gamma$. Let $C^{\pm}(z)$ be the analytic
functions defined by this integral in ${\sf D}$ and ${\sf D^c}$
respectively. By the Sokhotski-Plemelj formula, the jump of the
function $C(z)$ across the contour $\gamma$
is equal to $\p U(z, \bar z)$:
\beq\label{jump}
(C^+(z)-C^-(z))\bigr |_{z\in \gamma}= \p U(z, \bar z).
\eeq

\vspace{1ex}

As it follows from  the proof of Proposition
\ref{local-uniqueness}, $C^+(z)$ is the
generating function of the moments $t_k$:
$$
C^+(z)=\frac{1}{2\pi i}\oint_{\gamma}
\frac{U_{\zeta}(\zeta , \bar \zeta )\, d\zeta}{\zeta -z}=
\sum_{k\geq 1}kt_kz^{k-1}\,, \quad z\in {\sf D}.
$$
Similarly, $C^-(z)$ is the generating function of the set
of {\it complimentary moments} $v_k$:
$$
C^-(z)=\frac{1}{2\pi i}\oint_{\gamma}
\frac{U_{\zeta}(\zeta , \bar \zeta )\, d\zeta}{\zeta -z}=-
\frac{t_0}{z}
-\sum_{k\geq 1}v_kz^{-k-1}\,, \quad z\in {\sf D^c},
$$
which are given by the integrals
$$v_k=\frac{1}{2\pi i}\oint_{\gamma} z^{k}\partial U (z, \bar z) \, dz =
\frac{1}{\pi}\iint_{{\sf D}}\! z^{k}
\sigma  (z, \bar z) \, d^2 z  \,,\quad k\geq 1.
$$
They are functions of the moments $t_0, \{t_k \}, \{\bar t_k\}$.
It is also useful to introduce the logarithmic moment
$$v_0  =\frac{1}{\pi }\iint_{{\sf D}}\!
\log |z|^2 \sigma  (z, \bar z) \,  d^2 z.
$$

\begin{theorem} (cf. \cite{MWZ,WZ,Ztmf}) The following relations hold:
\beq\label{vk}
v_0=\p_0 F, \quad \quad v_k=\p_k F, \quad \quad \bar v_k=\bar \p_k F,
\quad \quad k\geq 1.
\eeq
\end{theorem}

\begin{proof} Applying Lemma \ref{harmonic1} to (\ref{F}), we get
(\ref{nablaF}).
The expansion of both sides in powers of $z$, $\bar z$ yields
(\ref{vk}).
\end{proof}

\begin{proposition}\label{2F}
Suppose that only a finite number of the moments $t_k$ are different
from 0. Then the tau-function (\ref{F}) can be represented as
\beq\label{Fv}
2F=-\, \frac{1}{\pi}\iint_{\sf D} U\sigma \, d^2z +t_0v_0+
\sum_{k\geq 1}(t_k v_k+\bar t_k \bar v_k).
\eeq
\end{proposition}

\begin{proof}
Set $\phi (z, \bar z)= -\, \frac{2}{\pi}\iint_{\sf D}
\log \left |z^{-1}-\zeta^{-1}\right |\sigma (\zeta , \bar \zeta )\,
d^2\zeta$, then under the assumption of the proposition
we have the expansion
$$
\phi (z, \bar z)=-U(z, \bar z)+t_0 \log |z|^2 +
\sum_{k\geq 1} (t_kz^k + \bar t_k \bar z^k)
$$
valid everywhere in ${\sf D}\! \setminus \! \{0\}$.
Substituting this into (\ref{F}) written in the form
$2F=\frac{1}{\pi}\iint_{\sf D}\phi \, \sigma \, d^2z$,
performing the termwise integration and using the
definition of the complimentary moments, we get (\ref{Fv}).
\end{proof}

\section{Conformal dynamics in a symmetric background}

\subsection{The general axially symmetric case}
The case when the background
density function $\sigma$ (and the function $U$)
is {\it axially symmetric},
i.e., depends only on $|z|$, is of a special interest.
Note that for any axially symmetric background it holds
$z\p U=\bar z\bar \p U$.
We denote the symmetric density by $\sigma (|z|^2)$. In this case
the conformal dynamics provides a symmetric solution to the
dispersionless 2DTL, so the function $F$ has the
explicit Taylor expansion obtained in section 2.
Indeed, the symmetry
implies that when all moments $t_k$ at $k\geq 1$ are
equal to 0,
the domain ${\sf D}$ is just a disk of radius $R$ such that
\beq\label{t0}
t_0=\int_{0}^{R^2}\!\! \sigma (x)\, dx, \quad \quad
v_0=\int_{0}^{R^2}\!\! \log x \, \sigma (x)\, dx
\eeq
and the complimentary moments $v_k$ with $k\geq 1$ are zero.
From (\ref{vk}) we see that
\beq\label{sym1}
v_k \bigr |_{t_0}= \p_{t_k}F\bigr |_{t_0} =0 \quad \mbox{for all
$k\geq 1$}
\eeq
so the solution is indeed symmetric in the sense of section 2.
One can see that $\p v_0 /\p t_0 =\log R^2$, i.e., $p=1/R$ as it should be.
We also have
\beq\label{F0}
F\bigr |_{t_0}
=\int_{0}^{R^2}\!\! \sigma (y)dy \int_{0}^{y}\log x \, \sigma (x)\,
dx, \quad \quad
f(t_0)=e^{F|_{t_0}''}=R^2(t_0),
\eeq
where $R^2$ should be understood as a function of $t_0$
implicitly given by (\ref{t0}).

The Green function at $t_k=0$ is
$$
2\pi G(z, \zeta )=\log \left |\frac{R^2(z-\zeta )}{z\bar \zeta -R^2}\right |
=\log \left |z^{-1}-\zeta ^{-1}\right | +\log R - {\rm Re}
\sum_{k\geq 1}\frac{R^{2k}}{kz^k \bar \zeta ^k}
$$
Comparison with (\ref{GF}) yields
\beq\label{sym1a}
\p_{t_j}\p_{t_k}F\bigr |_{t_0}=0\,, \quad
\p_{t_j}\p_{\bar t_k}F\bigr |_{t_0}=k R^{2k}\delta_{jk}
\eeq
in accordance with Lemma \ref{symmetric}.

Note that in the axially symmetric case the moment $v_0$ can be
represented as a contour integral:
$$
v_0=\frac{1}{2\pi i}\oint_{\gamma} \left (
\log |z|^2 \, \p U - z^{-1}U\right )dz.
$$

\subsection{Examples of symmetric solutions}

\subsubsection{The homogeneous density} An important example
is the homogeneous density
\beq\label{h1}
\sigma (z, \bar z)=(z\bar z)^{\alpha -1},
\quad \quad
U(z, \bar z)=\frac{1}{\alpha^2} \, (z\bar z)^{\alpha}
\eeq
with some $\alpha \in \mathbb{R}$. Assume that $\alpha >0$. In this case
$t_0=R^{2\alpha}/\alpha$, so $f(t_0)=(\alpha t_0)^{1/\alpha}$ and
$$
F\bigr |_{t_0}= \frac{1}{4\alpha^3}\left (
f^{2\alpha}\log f^{2\alpha} - 3 f^{2\alpha}\right )=
\frac{t_0^2}{2\alpha}  \log (\alpha t_0 )- \frac{3t_0^2}{4\alpha}
$$
(recall that for symmetric density $f=R^2$).

\begin{proposition} \label{alpha}
Assuming that only a finite number of the moments $t_k$ are different
from 0, the dispersionless tau-function for this solution
is quasi-homogeneous,
that is it obeys the relation
\beq\label{quasi-hom}
4\alpha F= -t_0^2 +2\alpha t_0 \p_0 F +
\sum_{k\geq 1} (2\alpha -k)
\left ( t_k \p_k F +\bar t_k \bar \p_k F\right ).
\eeq
\end{proposition}

\begin{proof}
We use equation (\ref{Fv}). In the integral term we write
$U\sigma = \bar \p (U\p U)-\p U \, \bar \p U$ and note that for the
particular function $U$ we have $U\sigma =\p U \, \bar \p U$, and also
$z\p U=\alpha U$, so
$U\sigma = \frac{1}{2}\, \bar \p (U\p U)$. This allows us to
transform the 2D integral to a contour integral:
$$
\frac{1}{\pi}\iint_{\sf D} U\, \sigma \, d^2 z =
\frac{1}{4\pi i\alpha}\oint_{\gamma} (z\p U)^2 \, \frac{dz}{z}\,.
$$
Now recall (\ref{jump}) and represent $\p U=C^+-C^-$ (on $\gamma$),
with $C^+$ being a polynomial.
Shrinking the integration contour to $\infty$, we obtain:
$$
\frac{1}{2\pi i}\oint_{\gamma} (z\p U)^2 \, \frac{dz}{z} =
t_0^2 +2\sum_{k\geq 1}kt_k v_k.
$$
Since the initial integral is obviously real, we conclude that
$\sum_{k\geq 1}kt_k v_k$ is a real quantity. The quasi-homogeneity
relation (\ref{quasi-hom}) follows.
\end{proof}

{\bf Remark.}
The case $\alpha =1$ ($\sigma (z, \bar z)=1$)
was considered in \cite{MWZ,WZ}. The Taylor
coefficients for this solution were found in \cite{N}.
In this case the conformal dynamics is identical to the Darcy law
for the motion of interface between viscous and non-viscous fluids
confined in the radial Hele-Shaw cell, assuming that there is no
surface tension at the interface. The normal
velocity of the interface $\gamma$ is proportional to the gradient
of the Green function of the Laplace operator.
This type of processes is known as
{\it Laplacian growth} (see, e.g., \cite{book,MWPT}).
As is mentioned in \cite{MWZ99}, the conformal dynamics
in the case $\alpha = 1/N$ for an integer $N>0$ can be mapped
to a Laplacian growth process in a sector with angle
$2\pi /N$ and periodic conditions at the boundary rays (a cone).
With proper modifications, the case $\alpha <0$ (and, in particular,
$\alpha =-1$) can be also
considered and mapped to the Laplacian growth in the compact
{\it interior} domain bounded by the curve $\gamma$.

\subsubsection{The solution yielding the Hurwitz numbers}
In this case
\beq\label{h2}
\sigma (z, \bar z)=\frac{1}{\beta \, z\bar z}\,,
\quad \quad
U(z, \bar z)=\frac{1}{2\beta} \Bigl [ \log \frac{z\bar z}{Q}\Bigr ]^2
\eeq
which can be formally regarded as a limit $\alpha \to 0$ of
(\ref{h1}) but the limit is tricky and this case deserves a separate
consideration. In this case
$$
F\bigr |_{t_0}= \frac{\beta t_0^3}{6}+ \frac{t_0^2}{2}\, \log Q,
\quad \quad
f(t_0)=Qe^{\beta t_0}.
$$

{\bf Remark.}
As is shown in \cite{Z}, the conformal dynamics
in this case can be mapped
to a Laplacian growth process in an infinite channel
with periodic conditions in the transverse direction,
i.e., on the surface of an infinite cylinder of radius $1/\beta$.

In this case the double integral (\ref{F}) diverges
because of the singularity at $0$.
The dispersionless tau-function in this case is given by
this integral with ${\sf D}$ changed to
${\sf D}\setminus {\sf B}(Q)$, where ${\sf B}(Q)$ is the disk
of radius $Q^{1/2}$ centered at the origin.

\begin{proposition} \label{alpha=0}(\cite{Z})
The dispersionless tau-function for this solution obeys
the following homogeneity relation:
$$
2F= \tau \p_{\tau}F +t_0\p_0 F +
\sum_{k\geq 1}
\left ( t_k \p_k F +\bar t_k \bar \p_k F\right ), \quad \quad
\tau \equiv 1/\beta \,.
$$
\end{proposition}

\begin{proof}
We have $\displaystyle{
t_k=\frac{1}{2\pi i k\beta}\oint_{\gamma} z^{-k}
\log \Bigl ( \frac{|z|^2}{Q}\Bigr ) \frac{dz}{z}}
$, $k\geq 1$,
$\displaystyle{
t_0=\frac{1}{2\pi i \beta}\oint_{\gamma}
\log \Bigl ( \frac{|z|^2}{Q}\Bigr ) \frac{dz}{z}}
$, so that the variables $\hat t_k=\beta t_k$ are
$\beta$-independent. As is seen from equation (\ref{F}) with
$\sigma$ as in (\ref{h2}), $F$ is of the form $F=\beta^{-2}\hat F$,
where $\hat F$ is $\beta$-independent. Therefore,
$$
F(\beta , \{t_k\}) =F(\beta , \{\hat t_k/\beta \})=\beta^{-2}\hat F=
\beta^{-2}F(1, \{\hat t_k\}).
$$
Taking the total $\beta$-derivative of the identity
$\beta^2 F(\beta , \{\hat t_k/\beta \})=F(1, \{\hat t_k\})$, we get:
$$
-2\beta^3 F-\beta^4 \p_{\beta}F +\beta^2 \Bigl ( \hat t_0 \, \p_0 F
+2 \,{\rm Re} \! \sum_{k\geq 1} \hat t_k \, \p_k F \Bigr )=0
$$
or
$$
2F=-\beta \p_{\beta}F +t_0 \, \p_0 F+
\sum_{k\geq 1} \Bigl (t_k \, \p_k F  +\bar t_k \, \bar \p_k F \Bigr ).
$$
\end{proof}

\noindent
The Taylor expansion of this function is given by (\ref{FF}),
with the Taylor coefficients being essentially the double Hurwitz
numbers for connected genus 0 coverings. The homogeneity property
can be explicitly seen from the Taylor expansion.

\begin{proposition} \label{cut-and-join}
The dispersionless tau-function for the solution
determined by the data (\ref{h2}) satisfies the relations
\beq\label{h3}
\begin{array}{l}
\displaystyle{\frac{\p F}{\p \log Q}=
\frac{t_0^2}{2}+\sum_{k\geq 1}kt_k \, \p_kF},
\\ \\
\displaystyle{\frac{\p F}{\p \beta}=\frac{t_0^3}{6}+
t_0\sum_{k\geq 1}kt_k \, \p_kF +\frac{1}{2}
\sum_{k,l\geq 1} \bigl ( kl t_kt_l \, \p_{k+l}F+
(k+l)t_{k+l}\, \p_k F \, \p_l F\bigr )}.
\end{array}
\eeq
\end{proposition}

\begin{proof}
The first formula is obvious from the structure of the
Taylor expansion (\ref{FF}). Another
proof, similar to that of Proposition \ref{alpha}, is given
in \cite{Z}. For the proof of the second formula we note that
in the case (\ref{h2}) $U \sigma = \frac{1}{3}\, \bar \p (U\p U)$
and thus
$$
\frac{1}{\pi}\iint_{{\sf D}\setminus {\sf B}(Q)}U\sigma \, d^2z =
\frac{\beta}{12\pi i}\oint_{\gamma}(z\p U)^3 \, \frac{dz}{z}\,.
$$
Proceeding as in the proof of Proposition \ref{alpha}, we obtain
$$
\frac{1}{12\pi i}\oint_{\gamma}(z\p U)^3 \, \frac{dz}{z}=
\frac{t_0^3}{6}+
t_0 \! \sum_{k\geq 1}kt_k \, v_k +\frac{1}{2}
\sum_{k,l\geq 1} \bigl ( kl t_kt_l  v_{k+l}+
(k+l)t_{k+l} v_k  v_l\bigr ).
$$
Then from Proposition \ref{2F} we have:
\beq\label{vn6b}
\begin{array}{ll}
2F &= \displaystyle{
t_0v_0 \! +\! \sum_{k\geq 1}(t_kv_k +\bar t_k \bar v_k)
-\frac{\beta t_0^3}{6}-\beta t_0 \! \sum_{k\geq 1} kt_kv_k}
\\ &\\
&\, -\,\,\, \displaystyle{\frac{\beta}{2}\sum_{k,l\geq 1}
\Bigl (kl t_k t_l v_{k+l}+ (k+l)t_{k\! +\! l}v_kv_l\Bigr ).}
\end{array}
\eeq
Combining this with the homogeneity
property (Proposition \ref{alpha=0}), we arrive at
the second formula in (\ref{h3}).
\end{proof}

\noindent
{\bf Remark.}
The double sum in the second equation  in (\ref{h3}) is the
genus 0 part of the celebrated cut-and-join operator \cite{GJ}, see also \cite{Takasaki12}.

\vspace{2ex}

\section*{Acknowledgments}
\addcontentsline{toc}{section}{Acknowledgements}

The work of S.Natanzon was supported in part
by RFBR grant 13-01-00755, by grant NSh-4850.2012.1 for support of
leading scientific schools, by Russian Government grant no.2010-220-01-077 (ag.no.11.G34.310005).
The work of A.Zabrodin was supported in part
by RFBR grant 11-02-01220, by joint RFBR grants 12-02-91052-CNRS,
12-02-92108-JSPS, by grant NSh-3349.2012.2 for support of
leading scientific schools and
by Ministry of Science and Education of Russian Federation
under contract 8498.
This study was carried out within ``The National Research University Higher School of Economics, Academic Fund Program in 2013-2014'', research grant No. 12-01-0122.


\begin{thebibliography}{References}

\addcontentsline{toc}{section}{References}




\bibitem{BEMS11}
G. Borot, B. Eynard, M. Mulase and B. Safnuk,
{\it A matrix model for simple Hurwitz numbers, and topological recursion},
J. Geom. Phys. {\bf 61} (2011) 522-540.




\bibitem{DMP} R. Dijkgraaf, G. Moore and R. Plesser, {\it
The partition function
of two-dimensional string theory}, Nucl. Phys. {\bf B394} (1993) 356-382.

\bibitem{EK} T. Eguchi and H. Kanno, {\it Toda lattice hierarchy
and the topological description of $c=1$ string theory},
Phys. Lett. {\bf B331} (1994) 330-334.

\bibitem{EV} P. Etingof and A. Varchenko,
{\it Why does the boundary of a round drop becomes a curve of
order four}, University Lecture Series, 3, American Mathematical Society, Providence, RI, 1992.


\bibitem{GJ} I.P. Goulden and D.M. Jackson, {\it Transitive
factorizations into transpositions and holomorphic mappings
on the sphere}, Proc. Amer. Math. Soc. {\bf 125} (1997) 51-60.


\bibitem{GJV} I. Goulden, D. Jackson and R. Vakil,
{\it Towards the geometry of double Hurwitz numbers},
Adv. Math. {\bf 198} (2005) 43-92.

\bibitem{book} B. Gustafsson, A. Vasil'ev, {\it Conformal and
Potential Analysis in Hele-Shaw Cells}, Birkh\"auser Verlag, 2006.

\bibitem{HOP} A. Hanany, Y. Oz and R. Plesser,
{\it Topological Landau-Ginzburg formulation and integrable structure
of 2d string theory}, Nucl. Phys. {\bf B425} (1994) 150-172.

\bibitem{Hur} A. Hurwitz, {\it Uber Riemann'shen Flachen mit gegebenen Verzweigungspunkten},  Math. Ann  {\bf 39} (1891) 1-61.

\bibitem{HC} A. Hurwitz and R. Courant, \textit{Vorlesungen \"uber allgemeine
Funktionentheorie und ellip\-tische Funktionen. Herausgegeben und
erg\"anzt durch einen Abschnitt \"uber geometrische Funktionentheorie,}
Springer-Verlag, 1964.

\bibitem{KKN} Yu. Klimov, A. Korzh and S. Natanzon
\textit{From 2D Toda hierarhy to conformal maps for domains of Riemann sphere}, Amer. Math. Soc. Trans. (2) {\bf 212} (2004) 207-218.


\bibitem{KKMWZ} I. Kostov, I. Krichever, M. Mineev-Weinstein,
P. Wiegmann and A. Zabrodin,
{\it $\tau$-function for analytic curves}, in: Random Matrix Models
and Their Applications, Math. Sci. Res. Inst. Publ. vol. 40,
Cambridge University Press, pp. 285-299, arXiv:hep-th/0005259.




\bibitem{KriW1}
I. Krichever, {\it The method of averaging for two dimensional integrable
equations}, Funct. Anal. Appl. {\bf 22} (1989) 200-213.



\bibitem{KriW} I. Krichever, {\it The $\tau$-function of the
universal Whitham hierarchy, matrix models and topological field
theories}, Comm. Pure Appl.  Math. {\bf 47} (1994) 437-475,
arXiv:hep-th/9205110.

\bibitem{KMZ05}
I. Krichever, A. Marshakov and A. Zabrodin,
{Integrable Structure of the Dirichlet
Boundary Problem in Multiply-Connected Domains},
Commun. Math. Phys. {\bf 259} (2005) 1-44.




\bibitem{Lando} S. Lando, {\it
Ramified coverings of the two-dimensional sphere and intersection theory
in spaces of meromorphic functions on algebraic curves},
Russ. Math. Surv. {\bf 57:3} (2002) 463-533.

\bibitem{Lando-Zvonkine [43 Lando]}S. Lando and D. Zvonkine,
{\it Counting ramified converings and intersection
theory on spaces of rational functions. I. Cohomology of Hurwitz spaces},
Mosc. Math. J. {\bf 7} (2007) 85-107.

\bibitem{MWZ} A. Marshakov, P. Wiegmann and A. Zabrodin,
{\it Integrable Structure of the Dirichlet
Boundary Problem in Two Dimensions},
Commun. Math. Phys. {\bf 227} (2002) 131-153.

\bibitem{MWWZ}
M. Mineev-Weinstein, P. Wiegmann and
A. Zabrodin, {\it Integrable structure of interface
dynamics}, Phys. Rev.
Lett. {\bf 84} (2000) 5106-5109, arXiv:nlin.SI/0001007.

\bibitem{MWPT}
M. Mineev-Weinstein, M. Putinar and R. Teodorescu,
{\it Random matrices in 2D, Laplacian growth and operator theory},
J. Phys. A: Math. Theor. {\bf 41} (2008) 263001,
arXiv:0805.0049.

\bibitem{MWZ99} M. Mineev-Weinstein and A. Zabrodin,
{\it Whitham-Toda hierarchy in the Laplacian growth problem},
Proceedings of the Workshop NEEDS-99 (Crete, Greece, June 1999),
J. Nonlin. Math. Phys. {\bf 8} (2001) 212-218.

\bibitem{N} S. Natanzon, {\it Towards an effectivization
of the Riemann theorem}, Ann. Global Anal. Geom. {\bf 28} (2005)
233-255.


\bibitem{N03} S. Natanzon, {\it Integrable systems and effectivization of the Riemann theorem about domains of the complex plane.},
Moscow. Math. J. {\bf 3(2)} (2003) 541-549.


\bibitem{N01} S. Natanzon, {\it Formulas for $A_n$ and $B_n$-solutions of WDVV equations}, J. Geom. Phys {\bf 39} (2001) 323-336.



\bibitem{Okounkov00} A. Okounkov, {\it Toda equations for
Hurwitz numbers}, Math. Res. Lett. {\bf 7} (2000) 447-453.


\bibitem{SSV}
S. Shadrin, M. Shapiro and A. Vainshtein,
{\it Chamber behavior of double Hurwitz numbers in genus 0},
Advances in Mathematics {\bf 217} (2008) 79-96.



\bibitem{Takasaki12} K. Takasaki,
{\it Generalized string equations for double Hurwitz numbers},
J. Geom. Phys. {\bf 62} (2012), 1135-1156.

\bibitem{TakTak1}
K. Takasaki and T. Takebe, {\it SDiff(2) Toda equation --
hierarchy, tau function and symmetries}, Lett. Math. Phys.
{\bf 23} (1991) 205-214.

\bibitem{TakTak} K. Takasaki and T. Takebe,
{\it Integrable hierarchies and dispersionless limit}, Rev. Math.
Phys. {\bf 7} (1995) 743-808.



\bibitem{Takstring} K. Takasaki, {\it Dispersionless Toda hierarchy
and two-dimensional string theory}, Commun. Math. Phys. {\bf 170}
(1995) 101-116.

\bibitem{Takh} L.Takhtajan, {\it Free bosons and tau-functions
for compact Riemann surfaces and closed smooth Jordan curves.
I. Current correlation functions}, Lett. Math. Phys. {\bf 56}
(2001) 181-228.

\bibitem{UenoTakasaki} K. Ueno and K. Takasaki, {\it Toda lattice
hierarchy}, Advanced Studies in Pure Math. {\bf 4} (1984) 1-95.

\bibitem{WZ} P. Wiegmann and A. Zabrodin,
{\it Conformal maps and integrable hierarchies},
Commun. Math. Phys. {\bf 213} (2000) 523-538.

\bibitem{Z} A. Zabrodin, {\it Laplacian growth in a channel
and Hurwitz numbers}, arXiv:1212.6729.

\bibitem{Ztmf} A. Zabrodin,  {\it The
dispersionless limit of the Hirota equations in some problems of
complex analysis}, Theor. Math. Phys.
{\bf 129} (2001) 1511-1525 (Teor. Mat. Fiz. {\bf 129} (2001) 239-257),
arXiv:math.CV/0104169.

\end{thebibliography}
\end{document}